\newfont{\bcb}{msbm10}
\newfont{\matb}{cmbx10}
\newfont{\got}{eufm10}

\documentclass[12pt]{amsart}
\usepackage{amsmath, amsthm, amscd, amsfonts, amssymb, latexsym, graphicx, color}
\usepackage[bookmarksnumbered, colorlinks, plainpages, hypertex]{hyperref}

\usepackage[cp1250]{inputenc}

\usepackage{amsmath,amsthm}
\usepackage{amssymb,latexsym}
\usepackage{enumerate}

\newtheorem{theorem}{Theorem}[section]
\newtheorem{lemma}[theorem]{Lemma}
\newtheorem{proposition}[theorem]{Proposition}
\newtheorem{corollary}[theorem]{Corollary}
\theoremstyle{definition}

\theoremstyle{remark}
\newtheorem{remark}[theorem]{Remark}
\numberwithin{equation}{section}

\begin{document}

\title[Algebraic geometry over rank one valued fields]{Some results
       of algebraic geometry\\ over Henselian rank one valued fields}

\author[Krzysztof Jan Nowak]{Krzysztof Jan Nowak}

%\footnotetext{Research partially supported by by KBN Grant No.\
%1P03A 00527.}



%\date{}

%\vspace{3ex}

\begin{abstract}
We develop geometry of affine algebraic varieties in $K^{n}$ over
Henselian rank one valued fields $K$ of equicharacteristic zero.
Several results are provided including: the projection $K^{n}
\times \mathbb{P}^{m}(K) \to K^{n}$ and blow-ups of the
$K$-rational points of smooth $K$-varieties are definably closed
maps; a descent property for blow-ups; curve selection for
definable sets; a general version of the \L{}ojasiewicz inequality
for continuous definable functions on subsets locally closed in
the $K$-topology and extending continuous hereditarily rational
functions, established for the real and $p$-adic varieties in our
joint paper with J.~Koll\'{a}r. The descent property enables
application of resolution of singularities and transformation to a
normal crossing by blowing up in much the same way as over the
locally compact ground field. Our approach relies on quantifier
elimination due to Pas and a concept of fiber shrinking for
definable sets, which is a relaxed version of curve selection. The
last three sections are devoted to the theory of regulous
functions and sets over such valued fields. Regulous geometry over
the real ground field $\mathbb{R}$ was developed by
Fichou--Huisman--Mangolte--Monnier. The main results here are
regulous versions of Nullstellensatz and Cartan's Theorems~A
and~B.
\end{abstract}

\maketitle

%\vspace{1ex}

\section{Introduction} In this paper, we develop geometry
of affine algebraic varieties in
$K^{n}$ over Henselian rank one valued fields $K$ of
equicharacteristic zero with valuation $v$, value group $\Gamma$,
valuation ring $R$ and residue field $\Bbbk$. Every rank one
valued field has a metric topology induced by its absolute value.
Examples of such fields are the quotient fields of the rings of
formal power series and of Puiseux series with coefficients from a
field $\Bbbk$ of characteristic zero as well as the fields of Hahn series
(maximally complete valued fields also called Malcev--Neumann
fields; cf.~\cite{Kap}):
$$ \Bbbk((t^{\Gamma})) := \left\{ f(t) = \sum_{\gamma \in \Gamma} \
   a_{\gamma}t^{\gamma} : \ a_{\gamma} \in \Bbbk, \ \text{supp}\,
   f(t) \ \text{is well ordered} \right\}. $$

%\vspace{1ex}

Let $X$ be a $K$-algebraic variety. We always assume that $X$ is
reduced but we allow it to be reducible. The set $X(K)$ of its
$K$-rational points ($K$-points for short) inherits from $K$ a
topology, called the $K$-topology. In this paper, we are going to
investigate continuous and differentiable functions $X(K) \to K$
that come from algebraic geometry and their zero sets. Therefore,
we shall (and may) most often assume that $X$ is an affine
$K$-variety such that $X(K)$ is Zariski dense in $X$. Throughout
the paper, by ''definable'' we shall mean ''definable with
parameters''.

%\vspace{1ex}

Several results concerning algebraic geometry over such ground
fields $K$ are established. Let $\mathcal{L}$ be the 3-sorted
language of Denef--Pas. We prove that the projection
$$ K^{n} \times \mathbb{P}^{m}(K) \to K^{n} $$
is an $\mathcal{L}$-definably closed map (Theorem~\ref{clo-th}).
Further, we shall draw several conclusions, including the theorem
that blow-ups of the $K$-points of smooth $K$-varieties are
definably closed maps (Corollary~\ref{clo-th-cor-3}), a descent
property for such blow-ups (Corollary~\ref{clo-th-cor-4}), curve
selection for $\mathcal{L}$-definable sets
(Proposition~\ref{GCSL}) and for valuative semialgebraic sets
(Proposition~\ref{CSL}) as well as a general version of the
\L{}ojasiewicz inequality for continuous $\mathcal{L}$-definable
functions on subsets locally closed in the $K$-topology
(Proposition~\ref{L-2}). Also given is a theorem on extending
continuous hereditarily rational functions over such ground fields
(Theorem~\ref{ext-th}), established for the real and $p$-adic
varieties in our joint paper~\cite{K-N} with J.~Koll\'{a}r. The
proof makes use of the descent property and the \L{}ojasiewicz
inequality. The descent property enables application of resolution
of singularities and transformation to a normal crossing by
blowing up (see \cite[Chap.~III]{Kol} for references and relatively
short proofs) in much the same way as over the locally compact ground
field. Our approach relies on quantifier elimination due to Pas
and on a certain concept of fiber shrinking for definable sets,
which is a relaxed version of curve selection. Note that this
paper comprises our two earlier preprints~\cite{Now1,Now2}.

\begin{remark}
This paper is principally devoted to geometry over rank one valued
fields (in other words, fields with non-archimedean absolute
value). Therefore, from Section~3 on, we shall most often assume
that so is the ground field $K$. Nevertheless, it is plausible
that the closedness theorem (Theorem~\ref{CD}) and curve selection
(Propositions~\ref{CSL} and~\ref{GCSL}) hold over arbitrary
Henselian valued fields.
\end{remark}

We should emphasize that our approach to the subject of this paper
is possible just because the language $\mathcal{L}$ in which we
investigate valued fields is not too rich; in particular, it does
not contain the inclusion language on the auxiliary sorts and the
only symbols of $\mathcal{L}$ connecting the sorts are two
functions from the main $K$-sort to the auxiliary $\Gamma$-sort
and $\Bbbk$-sort. Hence and by elimination of $K$-quantifiers, the
$\mathcal{L}$-definable subsets of the products of the two
auxiliary sorts are precisely finite unions of the Cartesian
products of sets definable in those two sorts. This allows us to
reduce our reasonings to an analysis of ordinary cells (i.e.\
fibers of a cell in the sense of Pas).

%\vspace{1ex}

The organization of the paper is as follows. In Section~2, we set
up notation and terminology including, in particular, the language
$\mathcal{L}$ of Denef--Pas and the concept of a cell. We recall
the theorems on quantifier elimination and on preparation cell
decomposition, due to Pas~\cite{Pa1}. Next we draw some
conclusions as, for instance, Corollary~\ref{def-fun} on definable
functions and Corollary~\ref{CD1} on certain
decompositions of definable sets. The former will be applied in Section~5
and the latter is crucial for our
proof of the closedness theorem (Theorem~\ref{clo-th}), which is
stated in Section~3 together with several direct corollaries,
including the descent property.
Section~4 gives a proof (being of algorithmic character) of this
theorem for the case where the value group $\Gamma$ is discrete.

%\vspace{1ex}

In Section~5, we study $\mathcal{L}$-definable functions of one
variable. A result playing an important role in the sequel is the
theorem on existence of the limit (Proposition~\ref{limit-th}).
Its proof makes use of Puiseux's theorem for the local ring of
convergent power series. In Section~6, we introduce a certain concept of fiber shrinking
for $\mathcal{L}$-definable sets (Proposition~\ref{FS}),
which is a relaxed version of curve selection. Section~7
provides a proof of the closedness theorem (Theorem~\ref{clo-th})
for the general case. This proof makes use of fiber shrinking and
existence of the limit for functions of one variable.

%\vspace{1ex}

In the subsequent three sections, some further conclusions from
the closedness theorems are drawn. Section~8 provides some
versions of curve selection: for arbitrary $\mathcal{L}$-definable
sets and for valuative semialgebraic sets. The next section is
devoted to a general version of the \L{}ojasiewicz inequality for
continuous $\mathcal{L}$-definable functions on subsets locally
closed in the $K$-topology (Proposition~\ref{L-2}). In Section~10,
the theorem on extending continuous hereditarily rational
functions (established for the real and $p$-adic varieties
in~\cite{K-N}) is carried over to the case where the ground field
$K$ is a Henseliam rank one valued field of equicharacteristic
zero. Let us mention that in real algebraic geometry applications
of continuous hereditarily rational functions and the extension
theorem, in particular, are given in the papers~\cite{K-1,K-2,K-3}
and~\cite{K-K}, which discuss rational maps into spheres and
stratified-algebraic vector bundles on real algebraic varieties.

%\vspace{1ex}

The last three sections are devoted to the theory of regulous
functions and sets over Henselian rank one valued fields of
equicharacteristic zero. Regulous geometry over the real ground
field $\mathbb{R}$ was developed by
Fichou--Huisman--Mangolte--Monnier~\cite{FHMM}. In Section~11 we
set up notation and terminology as well as provide basic results
about regulous functions and sets, including the noetherianity of
the constructible and regulous topologies. Those results are valid
over arbitrary fields with the density property. The next section
establishes a regulous version of Nullstellensatz
(Theorem~\ref{nullstellen}), valid over Henselian rank one valued
fields of equicharacteristic zero. The proof relies on the
\L{}ojasiewicz inequality (Proposition~\ref{L-2}). Also drawn are
several conclusions, including the existence of a one-to-one
correspondence between the radical ideals of the ring of regulous
functions and the closed regulous subsets, or one-to-one
correspondences between the prime ideals of that ring, the
irreducible regulous subsets and the irreducible Zariski closed
subsets (Corollaries~\ref{nullstellen-cor-1}
and~\ref{compari-th-cor-1}).

%\vspace{1ex}

Section~13 provides an exposition of the theory of quasi-coherent
regulous sheaves, which generally follows the approach given in
the real algebraic case by
Fichou--Huisman--Mangolte--Monnier~\cite{FHMM}. It is based on the
equivalence of categories between the category of
$\widetilde{\mathcal{R}}^{k}$-modules on the affine scheme
$\text{\rm Spec} \left( \mathcal{R}^{k}(K^{n}) \right)$ and the category of
$\mathcal{R}^{k}$-modules on $K^{n}$ which, in turn, is a direct consequence of the
one-to-one correspondences mentioned above. The main results here
are the regulous versions of Cartan's Theorems~A and~B. We also
establish a criterion for a continuous function on an affine
regulous subvariety to be regulous (Proposition~\ref{criterion}),
which relies on our theorem on extending continuous hereditarily
rational functions (Theorem~\ref{ext-th}).

%\vspace{1ex}

Note finally that the metric topology of a non-archimedean field
$K$ with a rank one valuation $v$ is totally disconnected. Rigid
analytic geometry (see e.g.~\cite{B-G-R} for its comprehensive
foundations), developed by Tate, compensates for this defect by
introducing sheaves of analytic functions in a Grothendieck
topology. Another approach is due to Berkovich~\cite{Ber}, who
filled in the gaps between the points of $K^{n}$, producing a
locally compact Hausdorff space (the analytification of
$K^{n}$), which contains the metric space $K^{n}$ as a dense
subspace whenever the ground field $K$ is algebraically closed. His
construction consists in replacing each point $x$ of a given
$K$-variety with the space of all rank one valuations on the
residue field $\kappa(x)$ that extend $v$. Further, the theory of stably
dominated types, developed by Hrushovski--Loeser~\cite{H-L}, deals with
non-archimedean fields with valuation of arbitrary rank and
generalizes that of tame topology for Berkovich spaces. Currently,
various analytic structures over Henselian rank one valued fields
are intensively investigated (see e.g.~\cite{C-Lip-0,C-Lip} for
more information, and~\cite{L-R} for the case of algebraically
closed valued fields).

\section{Quantifier elimination and cell decomposition} We begin
with quantifier elimination due to Pas in the language
$\mathcal{L}$ of Denef--Pas with three sorts: the valued field
$K$-sort, the value group $\Gamma$-sort and the residue field
$\Bbbk$-sort. The language of the $K$-sort is the language of
rings; that of the $\Gamma$-sort is any augmentation of the
language of ordered abelian groups (and $\infty$); finally, that
of the $\Bbbk$-sort is any augmentation of the language of rings.
We denote $K$-sort variables by $x,y,z,\ldots$, $\Bbbk$-sort
variables by $\xi,\zeta,\eta,\ldots$, and $\Gamma$-sort variables
by $k,q,r,\ldots$.

\vspace{1ex}

In the case of non-algebraically closed fields, passing to the
three sorts with additional two maps: the valuation $v$ and the
residue map, is not sufficient. Quantifier elimination due to Pas
holds for Henselian valued fields of equicharacteristic zero in
the above 3-sorted language with additional two maps: the
valuation map $v$ from the field sort to the value group, and a
map $\overline{ac}$ from the field sort to the residue field
(angular component map) which is multiplicative, sends $0$ to $0$
and coincides with the residue map on units of the valuation ring
$R$ of $K$.

\vspace{1ex}

Not all valued fields $K$ have an angular component map, but it exists
if $K$ has a cross section, which happens whenever $K$ is
$\aleph_{1}$-saturated (cf.~\cite[Chap.~II]{Ch}). Moreover, a valued field $K$ has an angular
component map whenever its residue field $\Bbbk$ is
$\aleph_{1}$-saturated (cf.~\cite[Corollary~1.6]{Pa2}). In general, unlike for
$p$-adic fields and their finite extensions, adding an angular
component map does strengthen the family of definable sets. For
both $p$-adic fields (Denef~\cite{De}) and Henselian
equicharacteristic zero valued fields (Pas~\cite{Pa1}), quantifier
elimination was established by means of cell decomposition and a
certain preparation theorem (for polynomials in one variable with
definable coefficients) combined with each other. In the latter
case, however, cells are no longer finite in number, but
parametrized by residue field variables. In the proof of the
closedness theorem, which is a fundamental tool for many
results of this paper, we may use an angular component map
because a given valued field can always be replaced with an
$\aleph_{1}$-saturated elementary extension.

\vspace{1ex}

Finally, let us mention that quantifier elimination based on the
sort $RV := K^{*}/(1+ \mathfrak{m}) \cup \{ 0 \}$ (where $K^{*} :=
K \setminus \{ 0 \}$ and $\mathfrak{m} $ is the maximal ideal of
the valuation ring $R$) was introduced by Besarab~\cite{Bes}. This
new sort binds together the value group and residue field into one
structure. In the paper~\cite[Section~12]{H-K}, quantifier
elimination for Henselian valued fields of equicharacteristic
zero, based on this sort, was derived directly from that by
A.~Robinson~\cite{Ro} for algebraically closed valued fields. Yet another, more
general result, including Henselian valued fields of mixed
characteristic, was achieved by Cluckers--Loeser~\cite{C-L} for
so-called b-minimal structures (from "ball minimal"); in the case
of valued fields, however, countably many sorts $RV_{n} :=
K^{*}/(1+ n \mathfrak{m}) \cup \{ 0 \}$, $n \in \mathbb{N}$, are
needed.

\vspace{1ex}

Below we state the theorem on quantifier elimination due to
Pas~\cite[Theorem~4.1]{Pa1}.

\begin{theorem}\label{QE}
Let $(K,\Gamma,\Bbbk)$ be a structure for the 3-sorted language
$\mathcal{L}$ of Denef--Pas. Assume that the valued field $K$ is
Henselian and of equicharacteristic zero. Then $(K,\Gamma,\Bbbk)$
admits elimination of $K$-quantifiers in the language
$\mathcal{L}$. \hspace*{\fill} $\Box$
\end{theorem}

%\vspace{2ex}

We immediately obtain the following

%\vspace{2ex}

{\begin{corollary}\label{QE-cor-1} The 3-sorted structure
$(K,\Gamma,\Bbbk)$ admits full elimination of quantifiers whenever
the theories of the value group $\Gamma$ and the residue field
$\Bbbk$ admit quantifier elimination in the languages of their
sorts. \hspace*{\fill} $\Box$
\end{corollary}

Below we prove another consequence of elimination of
$K$-quantifiers, which will be applied to the study of
definable functions of one variable in Section~5.

\begin{corollary}\label{def-fun}
Let $f:A \to K$ be an $\mathcal{L}$-definable function on a subset
$A$ of $K^{n}$. Then there is a finite partition of $A$ into
$\mathcal{L}$-definable sets $A_{i}$ and irreducible polynomials
$P_{i}(x,y)$, $i=1,\ldots,k$, such that for each $a \in A_{i}$ the
polynomial $P_{i}(a,y)$ in $y$ does not vanish and
$$ P_{i}(a,f(a))=0 \ \ \ \text{for all} \ \  a \in A_{i}, \ \
   i=1,\ldots,k. $$
\end{corollary}

%\vspace{1ex}

\begin{proof} By elimination of $K$-quantifiers,
the graph of $f$ is a finite union of sets $B_{i}$,
$i=1,\ldots,k$, defined by conditions of the form
$$ (v(f_{1}(x,y)),\ldots,v(f_{r}(x,y))) \in P, \ \  (\overline{ac}\, g_{1}(x,y),\ldots,\overline{ac}\, g_{s}(x,y))
   \in Q, $$
where $f_{i},g_{j} \in K[x,y]$ are polynomials, and $P$ and $Q$
are $\mathcal{L}$-definable subsets of $\Gamma^{r}$ and
$\Bbbk^{s}$, respectively. Each set $B_{i}$ is the graph of the
restriction of $f$ to an $\mathcal{L}$-definable subset $A_{i}$.
Since, for each point $a \in A_{i}$, the fibre of $B_{i}$ over $a$
consists of one point, the above condition imposed on angular
components includes one of the form $\overline{ac}\, g_{j}(x,y)=0$
or, equivalently, $g_{j}(x,y)=0$, for some $j=1,\ldots,s$, which
may depend on $a$, where the polynomial $g_{j}(a,y)$ in $y$ does
not vanish. This means that the set
$$ \{ (\overline{ac}\, g_{1}(x,y),\ldots,\overline{ac}\, g_{s}(x,y)):
   (x,y) \in B_{i} \} $$
is contained in the union of hyperplanes $\bigcup_{j=1}^{s} \{
\xi_{j}=0 \}$ and, furthermore, that for each point $a \in
A_{i}$ there is an index $j=1,\ldots,s$ such that the polynomial
$g_{j}(a,y)$ in $y$ does not vanish and $g_{j}(a,f(a))=0$.
Clearly, for any $j=1,\ldots,s$, this property of points $a \in
A_{i}$ is $\mathcal{L}$-definable. Therefore we can partition the
set $A_{i}$ into subsets each of which fulfils the condition
required in the conclusion with some irreducible factors of the
polynomial $g_{j}(x,y)$.
\end{proof}

%\vspace{2ex}

Recall now some notation concerning cell decomposition. Consider
an $\mathcal{L}$-definable subset $D$ of $K^{n} \times \Bbbk^{m}$,
three $\mathcal{L}$-definable functions
$$ a(x,\xi),b(x,\xi),c(x,\xi): D \to K $$
and a positive integer $\nu$. For each $\xi \in \Bbbk^{m}$ set
$$ C(\xi) := \{ (x,y) \in K^{n}_{x} \times K_{y}: \ (x,\xi) \in D,
$$
$$ v(a(x,\xi)) \lhd_{1} v((y-c(x,\xi))^{\nu}) \lhd_{2} v(b(x,\xi)), \
   \overline{ac} (y-c(x,\xi)) = \xi_{1} \}, $$
where $\lhd_{1},\lhd_{2}$ stand for $<, \leq$ or no condition in
any occurrence. If the sets $C(\xi)$, $\xi \in \Bbbk^{m}$, are
pairwise disjoint, the union
$$ C := \bigcup_{\xi \in \Bbbk^{m}} C(\xi) $$
is called a cell in $K^{n} \times K$ with parameters $\xi$ and
center $c(x,\xi)$; $C(\xi)$ is called a fiber of the cell $C$.

%\vspace{2ex}

\begin{theorem}\label{PCD} (Preparation Cell Decomposition, \cite[Theorem~3.2]{Pa1})
Let
$$ f_{1}(x,y),\ldots,f_{r}(x,y) $$
be polynomials in one variable $y$ with coefficients being
$\mathcal{L}$-definable functions on $K^{n}_{x}$. Then $K^{n}
\times K$ admits a finite partition into cells such that on each
cell $C$ with parameters $\xi$ and center $c(x,\xi)$ and for all
$i=1,\ldots,r$ we have:
$$ v(f_{i}(x,y)) = v \left( \tilde{f}_{i}(x,\xi)(y-c(x,\xi))^{\nu_{i}}
   \right), $$
$$ \overline{ac}\, f_{i}(x,y) = \xi_{\mu(i)}, $$
where $\tilde{f}_{i}(x,\xi)$ are $\mathcal{L}$-definable
functions, $\nu_{i} \in \mathbb{N}$ for all $i=1,\ldots,r$, and
the map $\mu: \{ 1,\ldots,r \} \to \{ 1,\ldots,m \}$ does not
depend on $x,y,\xi$.  \hspace*{\fill} $\Box$
\end{theorem}

\begin{remark}
The functions $f_{1}(x,y),\ldots,f_{r}(x,y)$ are said to be
prepared with respect to the variable $y$.
\end{remark}

Every divisible ordered group $\Gamma$ admits quantifier
elimination in the language $(<,+,-,0)$ of ordered groups.
Therefore it is not difficult to deduce from Theorems~\ref{QE}
and~\ref{PCD} the following

%\vspace{2ex}

\begin{corollary}\label{CD} (Cell decomposition) If, in addition,
the value group $\Gamma$ is divisible, then every
$\mathcal{L}$-definable subset $B$ of $K^{n} \times K$ is a finite
disjoint union of cells.  \hspace*{\fill} $\Box$
\end{corollary}

%\vspace{2ex}

Every archimedean ordered group $\Gamma$ (which of course may be regarded as a
subgroup of the additive group $\mathbb{R}$ of real numbers)
admits quantifier elimination in the Presburger language
$(<,+,-,0,1)$ with binary relation symbols $\equiv_{n}$ for
congruences modulo $n>1$, $n \in \mathbb{N}$, where $1$ denotes
the minimal positive element of $\Gamma$ if it exists or $1=0$
otherwise. Under the circumstances one can deduce in a similar manner the following

%\vspace{1ex}

\begin{corollary}\label{CD1}
If, in addition, the valuation $v$ is of rank $1$, then every
$\mathcal{L}$-definable subset $B$ of $K^{n} \times K$ is a finite
disjoint union of sets each of which is a subset
$$ F := \bigcup_{\xi \in \Bbbk^{m}} F(\xi) $$
of a cell
$$ C := \bigcup_{\xi \in \Bbbk^{m}} C(\xi) $$
determined by finitely many congruences:
$$ F(\xi) = \left\{ (x,y) \in C(\xi): \, v\left( f_{i}(x,\xi) (y -
   c(x,\xi))^{k_{i}} \right) \equiv_{M} 0, \ i=1,\ldots,s \right\}, $$
where $f_{i}$ are $\mathcal{L}$-definable functions, $k_{i} \in
\mathbb{N}$ for $i=1,\ldots,s$, and $M \in \mathbb{N}$, $M
>1$.   \hspace*{\fill} $\Box$
\end{corollary}

%\vspace{1ex}

\begin{remark}
Corollary~\ref{CD1} will be applied to establish the closedness
theorem (Theorem~\ref{clo-th}) in Section~7.
\end{remark}

%\vspace{5ex}

\section{Closedness theorem} In this paper we are interested mainly in geometry
over a Henselian rank one valued field of equicharacteristic zero.
From now on we shall assume (unless otherwise stated) that the
ground field $K$ is such a field. Below we state one of the basic
theorems, on which many other results of our paper rely.

%\vspace{2ex}

\begin{theorem}\label{clo-th} (Closedness theorem)
Let $D$ be an $\mathcal{L}$-definable subset of $K^{n}$. Then the
canonical projection
$$ \pi: D \times R^{m} \longrightarrow D  $$
is definably closed in the $K$-topology, i.e.\ if $B \subset D
\times R^{m}$ is an $\mathcal{L}$-definable closed subset, so is
its image $\pi(B) \subset D$.
\end{theorem}

%\vspace{2ex}

Observe that the $K$-topology is $\mathcal{L}$-definable whence
the above theorem is a first order property. Therefore it can be
proven using elementary extensions and thus one may assume that an
angular component map exists. We shall provide two different
proofs for this theorem. The first, given in Section~4, is valid
whenever the value group $\Gamma$ is discrete, and is based on a
procedure of algorithmic character. The other, given in Section~7,
is valid for the general case, and makes use of
Corollary~\ref{CD1} and fiber shrinking from Section~6 which, in
turn, relies on some results on $\mathcal{L}$-definable functions
of one variable from Section~5. When the ground field $K$ is
locally compact, the closedness theorem holds by a routine
topological argument. We immediately obtain five corollaries
stated below.

%\vspace{2ex}

\begin{corollary}\label{clo-th-cor-1}
Let $D$ be an $\mathcal{L}$-definable subset of $K^{n}$ and
$\,\mathbb{P}^{m}(K)$ stand for the projective space of dimension
$m$ over $K$. Then the canonical projection
$$ \pi: D \times \mathbb{P}^{m}(K) \longrightarrow D $$
is definably closed. \hspace*{\fill} $\Box$
\end{corollary}

\begin{corollary}\label{clo-th-cor-0}
Let $A$ be a closed $\mathcal{L}$-definable subset of
$\,\mathbb{P}^{m}(K)$ or $R^{m}$. Then every continuous
$\mathcal{L}$-definable map $f: A \to K^{n}$ is definably closed
in the $K$-topology.
\end{corollary}

%\vspace{2ex}

\begin{corollary}\label{clo-th-cor-2}
Let $\phi_{i}$, $i=0,\ldots,m$, be regular functions on $K^{n}$,
$D$ be an $\mathcal{L}$-definable subset of $K^{n}$ and $\sigma: Y
\longrightarrow K\mathbb{A}^{n}$ the blow-up of the affine space
$K\mathbb{A}^{n}$ with respect to the ideal
$(\phi_{0},\ldots,\phi_{m})$. Then the restriction
$$ \sigma: Y(K) \cap \sigma^{-1}(D) \longrightarrow D $$
is a definably closed quotient map.
\end{corollary}

%\vspace{2ex}

\begin{proof} Indeed,  $Y(K)$ can be regarded as a closed algebraic subvariety of
$K^{n} \times \mathbb{P}^{m}(K)$ and $\sigma$ as the canonical
projection.
\end{proof}

%\vspace{2ex}

Since the problem is local with respect to the target space, the
above corollary immediately generalizes to the case where the
$K$-variety $Y$ is the blow-up of a smooth $K$-variety $X$.

%\vspace{2ex}

\begin{corollary}\label{clo-th-cor-3}
Let $X$ be a smooth $K$-variety, $\phi_{i}$, $i=0,\ldots,m$,
regular functions on $X$, $D$ be an $\mathcal{L}$-definable subset
of $X(K)$ and $\sigma: Y \longrightarrow X$ the blow-up of the
ideal $(\phi_{0},\ldots,\phi_{m})$. Then the restriction
$$ \sigma: Y(K) \cap \sigma^{-1}(D) \longrightarrow D $$
is a definably closed quotient map.  \hspace*{\fill} $\Box$
\end{corollary}

%\vspace{2ex}

\begin{corollary}\label{clo-th-cor-4} (Descent property)
Under the assumptions of the above corollary, every continuous
$\mathcal{L}$-definable function
$$ g: Y(K) \cap \sigma^{-1}(D) \longrightarrow K $$
that is constant on the fibers
of the blow-up $\sigma$ descends to a (unique) continuous
$\mathcal{L}$-definable function $f: D \longrightarrow K$.
\hspace*{\fill} $\Box$
\end{corollary}

%\vspace{5ex}

\section{Proof of Theorem~\ref{clo-th} when the valuation is discrete}
The proof given in this section is of algorithmic character.
Through the transfer principle of Ax--Kochen--Ershov (see
e.g.~\cite{Ch}), it suffices to prove Theorem~\ref{clo-th} for the
case where the ground field $K$ is a complete, discretely valued
field of equicharacteristic zero. Such fields are, by virtue of
Cohen's structure theorem, the quotient fields $K=\Bbbk((t))$ of
formal power series rings $\Bbbk[[t]]$ in one variable $t$ with
coefficients from a field $\Bbbk$ of characteristic zero. The
valuation $v$ and the angular component $\overline{ac}$ of a
formal power series are the degree and the coefficient of its
initial monomial, respectively.

\vspace{1ex}

The additive group $\mathbb{Z}$ is an example of ordered
$Z$-group, i.e.\ an ordered abelian group with a (unique) smallest
positive element (denoted by $1$) subject to the following
additional axioms:
$$ \forall \, k \ \, k>0 \; \Rightarrow \; k \geq 1 $$
and
$$ \forall \, k \; \exists \: q \ \: \bigvee_{r=0}^{n-1} \; \, k = nq+r $$
for all integers $n>1$. The language of the value group sort will
be the Presburger language of ordered $Z$-groups, i.e.\ the
language of ordered groups $\{ <,+,-,0 \}$ augmented by $1$ and
binary relation symbols $\equiv_{n}$ for congruence modulo $n$
subject to the axioms:
$$ \forall \, k,r \ \; k \equiv_{n} r \; \Leftrightarrow \; \exists \,
   q \ \,  k-r=nq $$
for all integers $n>1$. This theory of ordered $Z$-groups has quantifier elimination and definable
Skolem (choice) functions. We can replace the above two countable axiom schemas
with universal ones after adding the unary function symbols $\left[ \frac{k}{n} \right]$ of one
variable $k$ for division by $n$ with remainder, which fulfil the
following postulates:
$$ \left[ \frac{k}{n} \right] = q \ \Leftrightarrow \ \bigvee_{r=0}^{n-1} \;
   k = nq+r $$
for all integers $n>1$. The theory of ordered $Z$-groups admits
therefore both quantifier elimination and universal axioms in the
Presburger language augmented by division with remainder. Thus
every definable function is piecewise given by finitely many terms
and, consequently, is piecewise linear.

\vspace{1ex}

In the residue field sort, we can add new relation symbols for all
definable sets and impose suitable postulates. This enables
quantifier elimination for the residue field in the augmented
language. In this fashion, we have full quantifier elimination in
the 3-sorted structure $(K, \mathbb{Z}, \Bbbk)$ with
$K=\Bbbk((t))$.

\vspace{1ex}

Now we can readily pass to the proof of Theorem~\ref{clo-th}
which, of course, reduces easily to the case $m=1$. So let $B$ be
an $\mathcal{L}$-definable closed (in the $K$-topology) subset of
$D \times R_{y} \subset K^{n}_{x} \times R_{y}$. It suffices to
prove that if $a$ lies in the closure of the projection $A :=
\pi(B)$, then there is a point $b \in B$ such that $\pi(b)=a$.

\vspace{1ex}

Without loss of generality, we may assume that $a=0$. Put
$$ \Lambda := \{ (v(x_{1}),\ldots,v(x_{n})) \in \mathbb{Z}^{n}:
   x=(x_{1},\ldots,x_{n}) \in A \}. $$
The set $\Lambda$ contains points all coordinates of which are
arbitrarily large, because the point $a=0$ lies in the closure of
$A$. Hence and by definable choice, $\Lambda$ contains a set
$\Lambda_{0}$ of the form
$$ \Lambda_{0} = \{ (k,\alpha_{2}(k),\ldots,\alpha_{n}(k)) \in \mathbb{N}^{n}:
   k \in \Delta \} \subset \Lambda, $$
where $\Delta \subset \mathbb{N}$ is an unbounded definable subset
and
$$ \alpha_{2},\ldots,\alpha_{n}: \Delta \longrightarrow \mathbb{N} $$
are increasing unbounded functions given by a term (because a
function in one variable given by a term is either increasing or
decreasing). We are going to recursively construct a point $b =
(0,w) \in B$ with $w \in R$ by performing the following procedure
of algorithmic character.

\vspace{1ex}

{\em Step 1.} Let
$$ \Xi_{1} := \{ (v(x_{1}),\ldots,v(x_{n}),v(y)) \in \Lambda_{0} \times \mathbb{N}:
   (x,y) \in B  \}, $$
and
$$ \beta_{1}(k) := \sup \left\{ l \in \mathbb{N}: (k,\alpha_{2}(k),\ldots,\alpha_{n}(k),l) \in
   \Xi_{1} \right\} \in \mathbb{N} \cup \{ \infty \}, \ \ k \in \Lambda_{0}. $$

If $\limsup_{k \rightarrow \infty} \ \beta_{1}(k) = \infty$, there
is a sequence $(x^{(\nu)},y^{(\nu)}) \in B$, $\nu \in \mathbb{N}$,
such that
$$ v(x_{1}^{(\nu)}),\ldots,v(x_{n}^{(\nu)}),v(y^{(\nu)}) \rightarrow
   \infty $$
when $\nu \rightarrow \infty$. Since the set $B$ is a closed
subset of $D \times  R_{y}$, we get
$$ (x^{(\nu)},y^{(\nu)}) \rightarrow 0 \in B \ \ \ \text{when} \ \
   \nu \rightarrow \infty, $$
and thus $w=0$ is the point we are looking for. Here the process
stops. Otherwise
$$ \Lambda_{1} \times \{ l_{1} \} \subset \Xi_{1} $$
for some infinite definable subset $\Lambda_{1}$ of $\Lambda_{0}$
and $l_{1} \in \mathbb{N}$. The set
$$ \{ (v(x_{1}),\ldots,v(x_{n});\overline{ac}(y)) \in \Lambda_{1} \times
   \Bbbk:
 (x,y) \in B, \ v(y)=l_{1} \} $$
is definable in the language $\mathcal{L}$. By full quantifier
elimination, it is given by a quantifier-free formula with
variables only from the value group $\Gamma$-sort and the residue
field $\Bbbk$-sort. Therefore there is a finite partitioning of
$\Lambda_{1}$ into definable subsets over each of which the fibres
of the above set are constant, because quantifier-free
$\mathcal{L}$-definable subsets of the product $\mathbb{Z}^{n}
\times \Bbbk$ of the two sorts are finite unions of the Cartesian
products of definable subsets in $\mathbb{Z}^{n}$ and in $\Bbbk$,
respectively. One of those definable subsets, say $\Lambda_{1}'$,
must be infinite. Consequently, for some $\xi_{1} \in \Bbbk$, the
set
$$ \Xi_{2} := \{ (v(x_{1}),\ldots,v(x_{n}),v(y-\xi_{1} t^{l_{1}})) \in
   \Lambda_{1}' \times \mathbb{N}: (x,y) \in B \} $$
contains points of the form $(k,l) \in \mathbb{N}^{n+1}$, where $k
\in \Lambda_{1}'$ and $l > l_{1}$.

\vspace{1ex}

{\em Step 2.} Let
$$ \beta_{2}(k) := \sup \left\{ l \in \mathbb{N}: (k,\alpha_{2}(k),\ldots,\alpha_{n}(k),l) \in
   \Xi_{2} \right\} \in \mathbb{N} \cup \{ \infty \}, \ \ k \in \Lambda_{1}'. $$
If $\limsup_{k \rightarrow \infty} \ \beta_{2}(k) = \infty$, there
is a sequence $(x^{(\nu)},y^{(\nu)}) \in B$, $\nu \in \mathbb{N}$,
such that
$$ v(x_{1}^{(\nu)}),\ldots,v(x_{n}^{(\nu)}),v(y^{(\nu)} -\xi_{1} t^{l_{1}}) \rightarrow
   \infty $$
when $\nu \rightarrow \infty$. Since the set $B$ is a closed
subset of $D \times  R_{y}$, we get
$$ (x^{(\nu)},y^{(\nu)}) \rightarrow (0,\xi_{1} t^{l_{1}}) \in B \ \ \ \text{when} \ \
   \nu \rightarrow \infty, $$
and thus $w=\xi_{1} t^{l_{1}}$ is the point we are looking for.
Here the process stops. Otherwise
$$ \Lambda_{2} \times \{ l_{2} \} \subset \Xi_{2} $$
for some infinite definable subset $\Lambda_{2}$ of $\Lambda_{1}'$
and $l_{2} > l_{1}$. Again, for some $\xi_{2} \in \Bbbk$, the set
$$ \Xi_{3} := \{ (v(x_{1}),\ldots,v(x_{n}),v(y-\xi_{1} t^{l_{1}} - \xi_{2} t^{l_{2}}))
   \in \Lambda_{2}' \times \mathbb{N}: (x,y) \in B \} $$
contains points of the form $(k,l) \in \mathbb{N}^{n+1}$, where $k
\in \Lambda_{2}'$, $\Lambda_{2}'$ is an infinite definable subset
of $\Lambda_{2}$ and $l > l_{2}$.

\vspace{1ex}

{\em Step 3} is carried out in the same way as the previous ones;
and so on.

\vspace{1ex}

In this fashion, the process either stops after a finite number of
steps and then yields the desired point $w \in R$ (actually, $w
\in \Bbbk[t]$) such that $(0,w) \in B$, or it does not stop and
then yields a formal power series
$$ w := \xi_{1}t^{l_{1}} + \xi_{2} t^{l_{2}} + \xi_{3} t^{l_{3}} +
   \ldots, \ \ \ 0 \leq l_{1} < l_{2} < l_{3} < \ldots  $$
such that for each $\nu \in \mathbb{N}$ there exists an element
$(x^{(\nu)},y^{(\nu)}) \in B$ for which
$$ v(y^{(\nu)} - \xi_{1} t^{l_{1}} - \xi_{2} t^{l_{2}} - \ldots -
   \xi_{\nu} t^{l_{\nu}}) \geq  l_{\nu} +1 \geq \nu , \ \ \
   v(x_{1}^{(\nu)}), \ldots ,v(x_{1}^{(\nu)}) \geq \nu. $$
Hence $v(y^{(\nu)} - w) \geq \nu$, and thus the sequence
$(x^{(\nu)},y^{(\nu)})$ tends to the point $b := (0,w)$ when $\nu$
tends to $\infty$. Since the set $B$ is a closed subset of $D
\times  R$, the point $b$ belongs to $B$, which completes the
proof. \hspace*{\fill} $\Box$

%\vspace{5ex}

\section{Definable functions of one variable}

Consider first a complete rank one valued field $L$. For every
non-negative integer $r$, let $L \{ x \}_{r}$ be the local ring of
all formal power series
$$ \phi(x) = \sum_{k=0}^{\infty} a_{k} x^{k} \in L[[ x ]] $$
in one variable $x$ such that $v(a_{k}) + kr$ tends to $\infty$
when $k \rightarrow \infty$; $L \{ x \}_{0}$ coincides with the
ring of restricted formal power series. Then the local ring
$$ L\{ x \} := \bigcup_{r=0}^{\infty} L\{ x \}_{r} $$
is Henselian, which can be directly deduced by means of the
implicit function theorem for restricted power series in several
variables (see~\cite[Chap.~III, \S~4]{Bour}, \cite{Fish} and
also~\cite[Chap.~I, \S~5]{G-R}).

\vspace{1ex}

We keep the assumption that the ground field $K$ is a Henselian
rank one valued field of equicharacteristic zero. Let $L$ be the
completion of the algebraic closure $\overline{K}$ of $K$.
Clearly, the Henselian local ring $L\{ x \}$ is closed under
division by the coordinate and power substitution. Therefore it
follows from our paper~\cite[Section~2]{Now} that Puiseux's
theorem holds for $L\{ x \}$. We still need an auxiliary lemma.

%\vspace{2ex}

\begin{lemma}\label{closed-lemma}
The field $K$ is a closed subspace of its algebraic closure
$\overline{K}$.
\end{lemma}

%\vspace{2ex}

\begin{proof} This follows directly from that the field $K$ is
algebraically maximal (as it is Henselian and finitely ramified;
see e.g.~\cite[Chap.~4]{E-Pre}), but can also be shown as follows.
Denote by $\text{cl}\, (E,F)$ the closure of a subset $E$ in $F$,
and let $\widehat{K}$ be the completion of $K$. We have
$$ \text{cl}\, (K,\overline{K}) = \text{cl}\, (K,L) \cap
   \overline{K} = \widehat{K} \cap \overline{K}. $$
Now, through the transfer principle of Ax-Kochen--Ershov (see
e.g.~\cite{Ch}), $K$ is an elementary substructure of
$\widehat{K}$ and, a fortiori, is algebraically closed in
$\widehat{K}$. Hence \ $\text{cl}\, (K,\overline{K}) = \widehat{K}
\cap \overline{K} = K$,\ as asserted.
\end{proof}

%\vspace{2ex}

Now consider an irreducible polynomial
$$ P(x,y) = \sum_{i=0}^{d} p_{i}(x)y^{i} \in K[x,y] $$
in two variables $x,y$ of $y$-degree $d \geq 1$. Let $Z$ be the
Zariski closure of its zero locus in $\overline{K} \times
\mathbb{P}^{1}(\overline{K})$. Performing a linear fractional
transformation over the ground field $K$ of the variable $y$, we
can assume that the fiber $\{ w_{1},\ldots,w_{s} \}$, $s \leq d$,
of $Z$ over $x=0$ does not contain the point at infinity, i.e.\
$w_{1},\ldots,w_{s} \in \overline{K}$. Then $p_{d}(0) \neq 0$ and
$p_{d}(x)$ is a unit in $L\{ x \}$. Via Hensel's lemma, we get the
Hensel decomposition
$$ P(x,y) = p_{d}(x) \cdot \prod_{j=1}^{s} P_{j}(x,y) $$
of $P(x,y)$ into polynomials
$$ P_{j}(x,y) = (y-w_{j})^{d_{j}} + p_{j1}(x) (y-w_{j})^{d_{j-1}} +
   \cdots + p_{jd_{j}}(x) \in L\{x \} [y-w_{j}] $$
which are Weierstrass with respect to $y-w_{j}$, $j=1,\ldots,s$,
respectively. By Puiseux's theorem, there is a neighbourhood $U$
of $0 \in \overline{K}$ such that the trace of $Z$ on $U \times
\overline{K}$ is a finite union of sets of the form
$$ Z_{\phi_{j}} = \{ (x^{q},\phi_{j}(x)): x \in U \} \ \ \text{with} \ \
   \phi_{j} \in L \{x \}, \ q \in \mathbb{N}, \ j=1,\ldots,s. $$
Obviously, for $j=1,\ldots,s$, the fiber of $Z_{\phi_{j}}$ over $x
\in U$ tends to the point $\phi_{j}(0)= w_{j}$ when $x \rightarrow
0$.

\vspace{1ex}

If $\phi_{j}(0) \in \overline{K} \setminus K$, it follows from
Lemma~\ref{closed-lemma} that
$$ Z_{\phi_{j}} \cap ((U \cap K) \times K) = \emptyset, $$
after perhaps shrinking the neighbourhood $U$.

Let us mention that if
$$  \phi_{j}(0) \in K  \ \ \ \text{and} \ \ \
    \phi_{j} \in L \{x \} \setminus \widehat{K} \{x \}, $$
then
$$ Z_{\phi_{j}} \cap ((U \cap K) \times K) = \{ (0,\phi_{j}(0)) \} $$
after perhaps shrinking the neighbourhood $U$. Indeed, let
$$ \phi_{j}(x) = \sum_{k=0}^{\infty} a_{k} x^{k} \in L[[ x ]] $$
and $p$ be the smallest positive integer with $a_{p} \in L
\setminus \widehat{K}$. Since $\widehat{K}$ is a closed subspace
of $L$, we get
$$ \sum_{k=p}^{\infty} a_{k} x^{k} = x^{p} \left( a_{p} + x \cdot
   \sum_{k=p+1}^{\infty} a_{k} x^{k-(p+1)} \right) \not \in \widehat{K} $$
for $x$ close enough to $0$, and thus the assertion follows.

\vspace{1ex}

Suppose now that an $\mathcal{L}$-definable function $f:A \to K$
satisfies the equation
$$ P(x,f(x))=0 \ \ \ \text{for} \ \ x \in A $$
and $0$ is an accumulation point of the set $A$. It follows
immediately from the foregoing discussion that the set $A$ can be
partitioned into a finite number of $\mathcal{L}$-definable sets
$A_{j}$, $j=1,\ldots,r$ with $r \leq s$, such that, after perhaps
renumbering of the fiber $\{ w_{1},\ldots,w_{s} \}$ of the set $\{
P(x,f(x))=0 \}$ over $x=0$, we have
$$ \lim_{x \rightarrow 0}\, f|A_{j}\, (x) = w_{j} \ \ \
   \text{for each} \ \ j=1,\ldots,r. $$

Hence and by Corollary~\ref{def-fun}, we immediately obtain the
following

%\vspace{2ex}

\begin{proposition}\label{limit-th} (Existence of the limit)
Let $f:A \to K$ be an $\mathcal{L}$-definable function on a subset
$A$ of $K$ and suppose $0$ is an accumulation point of $A$. Then
there is a finite partition of $A$ into $\mathcal{L}$-definable
sets $A_{1},\ldots,A_{r}$ and points $w_{1}\ldots,w_{r} \in
\mathbb{P}^{1}(K)$ such that
$$ \lim_{x \rightarrow 0}\, f|A_{j}\, (x) = w_{j} \ \ \ \text{for} \ \
   j=1,\ldots,r. $$
Moreover, there is a neighbourhood $U$ of $0$ such that each
definable set
$$ \{ (v(x), v(f(x))): \; x \in (A_{j} \cap U) \setminus \{0 \} \}
   \subset \Gamma \times (\Gamma \cup \ \{
   \infty \}),  \ j=1,\ldots,r, $$
is contained in an affine line with rational slope
$$ l = \frac{p_{j}}{q} \cdot k + \beta_{j}, \ j=1,\ldots,r, $$
with $p_{j},q \in \mathbb{Z}$, $q>0$, $\beta_{j} \in \Gamma$, or
in\/ $\Gamma \times \{ \infty \}$. \hspace*{\fill} $\Box$
\end{proposition}

%\vspace{2ex}
\begin{remark} Note that the first conclusion (existence of the
limit) could also be established via the lemma on the continuity
of roots of a monic polynomial (which can be found in
e.g.~\cite[Chap.~3, \S~4]{B-G-R}). Yet another approach for the
case of tame theories is provided in \cite[Lemma~2.20]{For}. The
second conclusion relies on Puiseux's parametrization.
\end{remark}

%\vspace{5ex}

\section{Fiber shrinking for definable sets}

Let $A$ be an $\mathcal{L}$-definable subset of $K^{n}$ with
accumulation point
$$ a = (a_{1},\ldots,a_{n}) \in K^{n} $$
and $E$ an
$\mathcal{L}$-definable subset of $K$ with accumulation point
$a_{1}$. We call an $\mathcal{L}$-definable family of sets
$$ \Phi = \bigcup_{t \in E} \ \{ t \} \times \Phi_{t} \subset A $$
an $\mathcal{L}$-definable $x_{1}$-fiber shrinking for the set $A$
at $a$ if
$$ \lim_{t \rightarrow a_{1}} \, \Phi_{t} = (a_{2},\ldots,a_{n}),
$$
i.e.\ for any neighbourhood $U$ of $(a_{2},\ldots,a_{n}) \in
K^{n-1}$, there is a neighbourhood $V$ of $a_{1} \in K$ such that
$\emptyset \neq \Phi_{t} \subset U$ for every $t \in V \cap E$, $t
\neq a_{1}$. When $n=1$, $A$ is itself a fiber shrinking for the
subset $A$ of $K$ at an accumulation point $a \in K$. This concept
is a relaxed version of curve selection. It is used in Sections~7
and~8 in the proofs of the closedness theorem and a certain
version of curve selection.

%\vspace{2ex}

\begin{proposition}\label{FS} (Fiber shrinking)
Every $\mathcal{L}$-definable subset $A$ of $K^{n}$ with
accumulation point $a \in K^{n}$ has, after a permutation of the
coordinates, an $\mathcal{L}$-definable $x_{1}$-fiber shrinking at
$a$.
\end{proposition}

%\vspace{2ex}

\begin{proof} We proceed with induction with respect to the
dimension of the ambient affine space $n$. The case $n=1$ is
trivial. So assuming the assertion to hold for $n$, we shall prove
it for $n+1$. We may, of course, assume that $a=0$. Let
$x=(x_{1},\ldots,x_{n+1})$ be coordinates in $K^{n}_{x}$.

\vspace{1ex}

If $0$ is an accumulation point of the intersections
$$ A \cap \{ x_{i} = 0 \}, \ \ i=1,\ldots,n +1, $$
we are done by the induction hypothesis. Thus we may assume that
the intersection
$$ A \cap  \bigcup_{i=1}^{n+1} \ \{ x_{i} = 0 \} = \emptyset  $$
is empty. Then the definable (in the $\Gamma$-sort) set
$$ P := \{ (v(x_{1}),\ldots,v(x_{n+1})) \in \Gamma^{n+1}: \; x
   \in A \} $$
has an accumulation point $(\infty,\ldots,\infty)$.

\vspace{1ex}

Since the $\Gamma$-sort admits quantifier elimination in the
language of ordered groups augmented by binary relation symbols
$\equiv_{n}$ for congruence modulo $n$, every definable subset of
$\Gamma^{n+1}$ is a finite union of subsets of semi-linear sets
contained in $\Gamma^{n+1}$ that are determined by a finite number
of congruences
\begin{equation}\label{cong}
  \sum_{j=1}^{n+1} \ r_{ij} \cdot k_{j} \equiv_{N} \alpha_{i}, \ \
  i=1,\ldots,s;
\end{equation}
here $N \in \mathbb{N}$, $N > 1$, $r_{ij} \in \mathbb{Z}$,
$\alpha_{i} \in \Gamma$ for $i=1,\ldots,s$, $j=1,\ldots,n+1$.

\vspace{1ex}

Consequently, there exists a semi-linear subset $P_{0}$ of
$\mathbb{R}^{n+1}$ given by finitely many linear equations and
inequalities with integer coefficients and with constant terms
from $\Gamma$ such that the subset $P_{1}$ of $P_{0} \cap
\Gamma^{n+1}$ determined by congruences of the form ~\ref{cong} is
contained in $P$ and has an accumulation point
$(\infty,\ldots,\infty)$. Therefore there exists an affine
semi-line
$$ L := \left\{ \left( r_{1} \cdot k + \gamma_{1}, \ldots,
   r_{n+1} \cdot k + \gamma_{n+1} \right) :
   \; k \in \Gamma, \, k \geq 0 \right\}, $$
where $r_{1},\ldots,r_{n+1}$ are positive integers, passing
through a point
$$ (\gamma_{1},\ldots,\gamma_{n+1}) \in P_{1} \subset \Gamma^{n+1} $$
and contained in $P_{0}$. It is easy to check that the set
$$ L_{0} := \{ (\gamma_{1} + rr_{1}N, \ldots, \gamma_{n+1} + r r_{n+1}N) :
   \ r \in \mathbb{N} \} \subset P_{1} $$
is contained in $P_{1}$. Then
$$ \Phi := \{ x \in A:
   (v(x_{1}),\ldots,v(x_{n+1})) \in L_{0} \} $$
is an $\mathcal{L}$-definable $x_{1}$-fiber shrinking for the set
$A$ at $0$. This finishes the proof.
\end{proof}

\section{Proof of Theorem~\ref{clo-th} for the general case}
The proof reduces easily to the case $m=1$. We must show that if
$B$ is an $\mathcal{L}$-definable subset of $D \times R$ and a
point $a$ lies in the closure of $A := \pi(B)$, then there is a
point $b$ in the closure of $B$ such that $\pi(b)=a$. We may
obviously assume that $a = 0 \not \in A$. By Proposition~\ref{FS},
there exists, after a permutation of the coordinates, an
$\mathcal{L}$-definable $x_{1}$-fiber shrinking $\Phi$ for $A$ at
$a$:
$$ \Phi = \bigcup_{t \in E} \ \{ t \} \times \Phi_{t} \subset A, \ \ \
   \lim_{t \rightarrow 0} \, \Phi_{t} = 0; $$
here $E$ is the canonical projection of $A$ onto the $x_{1}$-axis.
Put
$$ B^{*} := \{ (t,y)  \in K \times R: \ \exists \; u \in \Phi_{t}
   \ (t,u,y) \in B \};  $$
it easy to check that if a point $(0,w) \in K^{2}$ lies in the
closure of $B^{*}$, then the point $(0,w) \in K^{n+1}$ lies in the
closure of $B$. The problem is thus reduced to the case $n=1$ and
$a=0 \in K$.

\vspace{1ex}

By Corollary~\ref{CD1}, we can assume that $B$ is a subset $F$ of
a cell $C$
$$ F \subset C \subset K_{x} \times R \subset K_{x} \times K_{y} $$
of the form
$$ F(\xi) := \{ (x,y) \in K_{x} \times K_{y}:\ (x,\xi) \in D,
$$
$$ v(a(x,\xi)) \lhd_{1} v((y-c(x,\xi))^{\nu}) \lhd_{2} v(b(x,\xi)), \
   \overline{ac} (y-c(x,\xi)) = \xi_{1}, $$
$$ v\left( f_{i}(x,\xi) (y -
   c(x,\xi))^{k_{i}} \right) \equiv_{M} 0, \ i=1,\ldots,s \}. $$
But the set
$$ \{ (v(x),\xi) \in \Gamma \times
   \Bbbk^{m}: \;  \exists \, y \in R \ \, (x,y) \in F(\xi) \} $$
is an $\mathcal{L}$-definable subset of the product $\Gamma \times
\Bbbk^{m}$ of the two sorts, which is, by elimination of
$K$-quantifiers, a finite union of the Cartesian products of
definable subsets in $\Gamma$ and in $\Bbbk^{m}$, respectively. It
follows that $0$ is an accumulation point of the projection
$\pi(F(\xi'))$ of the fiber $F(\xi')$ for a parameter $\xi' \in
\Bbbk^{m}$. We are thus reduced to the case where $B$ is the fiber
$F(\xi')$ of the set $F$ for a parameter $\xi'$. For simplicity,
we abbreviate $c(x,\xi'),a(x,\xi'),b(x,\xi')$ and $f_{i}(x,\xi')$
to $c(x),a(x),b(x)$ and $f_{i}(x)$, $i=1,\ldots,s$. Denote by $E
\subset K$ the domain of these functions; then $0$ is an
accumulation point of $E$.

\vspace{1ex}

In the statement of Theorem~\ref{clo-th}, we may equivalently
replace $R$ with the projective line $\mathbb{P}^{1}(K)$, because
the latter is the union of two open and closed charts biregular to
$R$. By Proposition~\ref{limit-th}, we can thus assume that the
limits, say $c(0), a(0),b(0)$,$f_{i}(0)$ of $c(x),a(x),b(x)$,
$f_{i}(x)$ ($i=1,\ldots,s$) when $x \rightarrow 0\,$ exist in
$\mathbb{P}^{1}(K)$ and, moreover, there is a neighbourhood $U$ of
$0$ such that, each definable set
$$ \{ (v(x), v(f_{i}(x))): \; x \in (E \cap U) \setminus \{0 \} \}
   \subset \Gamma \times (\Gamma \cup \ \{
   \infty \}), \ i=1,\ldots,s,  $$
is contained in an affine line with rational slope
\begin{equation}\label{affine}
l = \frac{p_{i}}{q} \cdot k + \beta_{i}, \ i=1,\ldots,s,
\end{equation}
with $p_{i},q \in \mathbb{Z}$, $q>0$, $\beta_{i} \in \Gamma$, or
in\/ $\Gamma \times \{ \infty \}$.

\vspace{1ex}

Performing a linear fractional transformation of the coordinate
$y$, we get
$$ c(0), a(0),b(0) \in K. $$
The role of the center $c(x)$ is immaterial. We can assume,
without loss of generality, that it vanishes, $c(x) \equiv 0$, for
if a point $b = (0,w) \in K^{2}$ lies in the closure of the cell
with zero center, the point $(0, w + c(0))$ lies in the closure of
the cell with center $c(x)$.

\vspace{1ex}

When $\lhd_{1}$ occurs and $a(0)= 0$, the set $F(\xi')$ is itself
an $x$-fiber shrinking at $(0,0)$ and the point $b=(0,0)$ is an
accumulation point of $B$ lying over $a=0$, as desired.

\vspace{1ex}

So suppose that either only $\lhd_{2}$ occurs or $\lhd_{1}$ occurs
and $a(0) \neq 0$. By elimination of $K$-quantifiers, the set
$v(E)$ is a definable subset of $\Gamma$. The value group $\Gamma$
admits quantifier elimination in the language of ordered groups
augmented by symbols $\equiv_{n}$ for congruences modulo $n$, $n
\in \mathbb{N}$, $n > 1$ (cf.\ Section~2). Therefore the set
$v(E)$ is of the form
\begin{equation}\label{vE}
v(E) = \{ k \in (\alpha, \infty) \cap \Gamma: \ m_{j} k \equiv_{N}
   \gamma_{j}, \ j=1,\ldots,t \},
\end{equation}
where $\alpha, \gamma_{j} \in \Gamma$, $m_{j} \in \mathbb{N}$ for
$j=1,\ldots,t$.

\vspace{1ex}

Now, take an element $(u,w) \in F(\xi')$ with $u \in (E \cap U)
\setminus \{ 0 \}$. By equality~\ref{vE}, there is a point $x_{r}
\in E$, $r \in \mathbb{N}$, with
$$ v(u_{r}) = v(u) + rq M N. $$
By equality~\ref{affine}, we get
$$ v(f_{i}(u_{r})) = v(f_{i}(u)) + rp_{i} M N, \ \ i=1,\ldots,s. $$
Hence
\begin{equation}\label{vf}
   v\left( f_{i}(u_{r})w^{k_{i}} \right) = v(f_{i}(u_{r})) + k_{i}
   v(w) =
\end{equation}
$$  = v(f_{i}(u)) + rp_{i} M N + k_{i} v(w) =
    v\left( f_{i}(u)w^{k_{i}} \right) + rp_{i} M N  \equiv_{M} 0.
$$
Of course, after shrinking the neighbourhood $U$, we may assume
that $v(a(x)) = v(a(0)) < \infty$ for all $x \in (E \cap U)
\setminus \{ 0 \}$. Consequently,
$$ v(a(u_{r})) \lhd_{1} v(w^{\nu}) \lhd_{2} v(b(u_{r})). $$
Hence and by \ref{vf}, we get $(u_{r},w) \in F(\xi')$. Since
$u_{r}$ tends to $0 \in K$ when $r \rightarrow \infty$, the point
$(0,w)$ is an accumulation point of $F(\xi')$ lying over $0 \in
K$, which completes the proof of the closedness theorem.
 \hspace*{\fill} $\Box$

%\vspace{5ex}

\section{Curve selection}
We now pass to curve selection over non-locally compact ground
fields under study. While the real version of curve selection goes
back to the papers~\cite{B-Car,Wal} (see
also~\cite{Loj,Miln,BCR}), the $p$-adic one was achieved in the
papers~\cite{S-Dries,De-Dries}. Before proving a general version
for $\mathcal{L}$-definable sets, we give a version for valuative
semialgebraic sets. Our approach relies on resolution of
singularities, which was already suggested by Denef--van den
Dries~\cite{De-Dries} in the remark after Theorem~3.34.

\vspace{1ex}

By a valuative semialgebraic subset of $K^{n}$ we mean a (finite)
Boolean combination of elementary valuative semialgebraic subsets,
i.e.\ sets of the form
$$ \{ x \in K^{n}: \ v(f(x)) \leq v(g(x)) \}, $$
where $f$ and $g$ are regular functions on $K^{n}$. We call a map
$\varphi$ semialgebraic if its graph is a valuative semialgebraic
set.

%\vspace{2ex}

\begin{proposition}\label{CSL}
Let $A$ be a valuative semialgebraic subset of $K^{n}$. If a point
$a \in K^{n}$ lies in the closure (in the $K$-topology) of $A
\setminus \{ a \}$, then there is a semialgebraic map $\varphi : R
\longrightarrow K^{n}$ given by restricted power series such that
$$ \varphi(0)=a \ \ \ {\text and} \ \ \  \varphi( R
   \setminus \{ 0 \}) \subset A \setminus \{ a \}. $$
\end{proposition}

\begin{proof} It is easy to check that every valuative semialgebraic set is a
finite union of basic valuative semialgebraic sets, i.e.\ sets of
the form
$$ \{ x \in K^{n}: \ v(f_{1}(x)) \lhd_{1} v(g_{1}(x)), \: \ldots \: ,
   \ v(f_{r}(x)) \lhd_{r} v(g_{r}(x)) \}, $$
where $f_{1},\ldots,f_{r},g_{1},\ldots,g_{r}$ are regular
functions and $\lhd_{1},\ldots,\lhd_{r}$ stand for $\leq$ or $<$.
We may assume, of course, that $A$ is a set of this form and
$a=0$. Take a finite composite
$$ \sigma: Y \longrightarrow K\mathbb{A}^{n} $$
of blow-ups along smooth centers such that the pull-backs of the
coordinates $x_{1},\ldots,x_{n}$ and the pull-backs
$$ f_{1}^{\sigma} := f_{1} \circ \sigma,\, \ldots \, ,f_{r}^{\sigma}:= f_{r} \circ \sigma \ \ \ \text{and}
   \ \ \  g_{1}^{\sigma}:= g_{1} \circ \sigma,\, \ldots \, , g_{r}^{\sigma} :=g_{r} \circ \sigma $$
are normal crossing divisors ordered with respect to divisibility
relation, unless they vanish. Since the restriction $\sigma: Y(K)
\longrightarrow K^{n}$ is definably closed
(Corollary~\ref{clo-th-cor-3}), there is a point $b \in Y(K) \cap
\sigma^{-1}(a)$ which lies in the closure of the set
$$ B := Y(K) \cap \sigma^{-1}(A \setminus \{ a \}). $$
Further, we get
$$ Y(K) \cap \sigma^{-1}(A) = \{ v(f_{1}^{\sigma}(y)) \lhd_{1} v(g_{1}^{\sigma}(y)) \}
   \ \cap \ \ldots
   \ \cap \ \{  v(f_{r}^{\sigma}(y)) \lhd_{r} v(g_{r}^{\sigma}(y)) \}, $$
and thus $\sigma^{-1}(A)$ is in suitable local coordinates
$y=(y_{1},\ldots,y_{n})$ near $b=0$ a finite intersection of sets
of the form
$$ \{ v(y^{\alpha}) \leq v(u(y)) \}, \  \{ v(u(y)) \leq v(y^{\alpha}) \},
   \ \{ v(y^{\beta}) < \infty \} \  \text{or} \  \{ \infty =
   v(y^{\gamma}) \}, $$
where $\alpha,\beta,\gamma \in \mathbb{N}^{n}$ and $u(y)$ is a
regular, nowhere vanishing function.

The first case cannot occur because $b=0$ lies in the closure of
$B$; the second case holds in a neighbourhood of $b$; the third
and fourth cases are equivalent to $y^{\beta} \neq 0$ and
$y^{\gamma} =0$, respectively. Consequently, since the pull-backs
of the coordinates $x_{1},\ldots,x_{n}$ are monomial divisors too,
$B$ contains the set $(R \setminus \{ 0 \}) \cdot c$ when $c \in
B$ is a point sufficiently close to $b=0$. Then the map
$$ \varphi: R \longrightarrow K^{n}, \ \ \ \varphi(z) = \sigma (z
   \cdot c) $$
has the desired properties.
\end{proof}

We now pass to the general version of curve selection for
$\mathcal{L}$-definable sets.

\begin{proposition}\label{GCSL}
Let $A$ be an $\mathcal{L}$-definable set subset of $K^{n}$. If a
point $a \in K^{n}$ lies in the closure (in the $K$-topology) of
$A \setminus \{ a \}$, then there exist a semialgebraic map
$\varphi : R \longrightarrow K^{n}$ given by restricted power
series and an $\mathcal{L}$-definable subset $E$ of $R$ with
accumulation point $0$ such that
$$ \varphi(0)=a \ \ \ {\text and} \ \ \  \varphi( E
   \setminus \{ 0 \}) \subset A \setminus \{ a \}. $$
\end{proposition}

\begin{proof}
We proceed with induction with respect to the dimension of the
ambient space $n$. The case $n=1$ being evident, suppose $n>1$. By
elimination of $K$-quantifiers, similarly as in Section~2, the set
$A \setminus \{ a \}$ is a finite union of sets defined by
conditions of the form
$$ (v(f_{1}(x)),\ldots,v(f_{r}(x))) \in P, \ \ (\overline{ac}\, g_{1}(x),\ldots,\overline{ac}\, g_{s}(x)) \in Q,
$$
where $f_{i},g_{j} \in K[x]$ are polynomials, and $P$ and $Q$
are definable subsets of $\Gamma^{r}$ and $\Bbbk^{s}$,
respectively (since $x=0$ iff $\overline{ac}\, x =0$). Thus we may
assume that $A$ is such a set and, of course, that $a=0$.

\vspace{1ex}

Again, take a finite composite
$$ \sigma: Y \longrightarrow K\mathbb{A}^{n} $$
of blow-ups along smooth centers such that the pull-backs
$$ f_{1}^{\sigma},\ldots, f_{r}^{\sigma} \ \ \ \text{and}
   \ \ \  g_{1}^{\sigma},\ldots, g_{r}^{\sigma} $$
are normal crossing divisors unless they vanish. Since the
restriction $\sigma: Y(K) \longrightarrow K^{n}$ is definably
closed (Corollary~\ref{clo-th-cor-3}), there is a point $b \in
Y(K) \cap \sigma^{-1}(a)$ which lies in the closure of the set
$$ B := Y(K) \cap \sigma^{-1}(A \setminus \{ a \}). $$
Take local coordinates $y_{1}.\ldots,y_{n}$ near $b$ in which
$b=0$ and every pull-back above is a normal crossing. We shall
first select a semialgebraic map $\psi : R \longrightarrow Y(K)$
given by restricted power series and an $\mathcal{L}$-definable
subset $E$ of $R$ with accumulation point $0$ such that
$$ \psi(0)=b \ \ \ {\text and} \ \ \  \psi( E
   \setminus \{ 0 \}) \subset B. $$

\vspace{1ex}

Since the valuation map and the angular component map composed
with a continuous function are locally constant near any point at
which this function does not vanish, the conditions which describe
the set $B$ near $b$ are of the form
$$ (v(y_{1}),\ldots,v(y_{n})) \in \widetilde{P}, \ \ (\overline{ac}\, y_{1},\ldots,\overline{ac}\, y_{n}) \in
  \widetilde{Q}, $$
where $\widetilde{P}$ and $\widetilde{Q}$ are definable subsets of
$\Gamma^{n}$ and $\Bbbk^{n}$, respectively.

\vspace{1ex}

The set $B_{0}$ determined by the conditions
$$ (v(y_{1}),\ldots,v(y_{n})) \in \widetilde{P}, $$
$$ (\overline{ac}\, y_{1},\ldots,\overline{ac}\, y_{n}) \in \widetilde{Q} \, \cap
   \, \bigcup_{i=1}^{n} \, \{ \xi_{i}=0 \}, $$
is contained near $b$ in the union of hyperplanes $\{ y_{i}=0 \}$,
$i=1,\ldots,n$. If $b$ is an accumulation point of the set
$B_{0}$, then the desired map $\psi$ exists by the induction
hypothesis. Otherwise $b$ is an accumulation point of the set
$B_{1} := B \setminus B_{0}$.

\vspace{1ex}

Analysis from the proof of Proposition~\ref{FS} (fiber shrinking)
shows that the congruences describing the definable subset
$\widetilde{P}$ of $\Gamma^{n}$ are not an essential obstacle to
finding the desired map $\psi$, but affect only the definable
subset $E$ of $R$. Neither are the conditions
$$ \widetilde{Q} \, \setminus \, \bigcup_{i=1}^{n} \, \{ \xi_{i}=0 \}$$
imposed on the angular components of the coordinates
$y_{1},\ldots,y_{n}$, because then none of them vanishes.
Therefore, in order to select the map $\psi$, we must first of all
analyze the linear conditions (equalities and inequalities)
describing the set $\widetilde{P}$.

\vspace{1ex}

The set $\widetilde{P}$ has an accumulation point
$(\infty,\ldots,\infty)$ as $b=0$ is an accumulation point of $B$.
We see, similarly as in the proof of Proposition~\ref{FS} (fiber
shrinking), that $\widetilde{P}$ contains a definable subset of a
semi-line
$$ L := \left\{ \left( r_{1} \cdot k + \gamma_{1}, \ldots,
   r_{n} \cdot k + \gamma_{n} \right) :
   \; k \in \Gamma, \, k \geq 0 \right\}, $$
where $r_{1},\ldots,r_{n}$ are positive integers, passing through
a point
$$ \gamma_{1},\ldots,\gamma_{n} \in \widetilde{P} \subset
   \Gamma^{n}; $$
clearly, $(\infty,\ldots,\infty)$ is an accumulation point of that
definable subset of $L$.

\vspace{1ex}

Now, take some elements
$$ (\xi_{1},\ldots,\xi_{n}) \in \widetilde{Q} \, \setminus
   \, \bigcup_{i=1}^{n} \, \{ \xi_{i}=0 \} $$
and next some elements $w_{1},\dots,w_{n} \in K$ for which
$$ v(w_{1})=\gamma_{1},\ldots,v(w_{n})=\gamma_{n} \ \ \ \text{and}
   \ \ \ \overline{ac}\, w_{1} = \xi_{1}, \ldots, \overline{ac}\, w_{n}
   = \xi_{n}. $$
There exists an $\mathcal{L}$-definable subset $E$ of $R$ which is
determined by some congruences imposed on $v(t)$ (as in the proof
of Proposition~\ref{FS}) and the conditions $\overline{ac}\, t =1$
such that the subset
$$ F := \left\{ \left( w_{1} \cdot t^{r_{1}}, \ldots, w_{n} \cdot t^{r_{n}} \right):
   \; t \in E \right\} $$
of the arc
$$ \psi: R \to Y, \ \ \psi(t) = \left( w_{1} \cdot t^{r_{1}}, \ldots, w_{n} \cdot t^{r_{n}}
   \right) $$
is contained in $B_{1}$. Then $\varphi := \sigma \circ \psi$ is
the map we are looking for. This completes the proof.
\end{proof}

%\vspace{5ex}

\section{\L{}ojasiewicz inequality}
In this section we provide certain general versions of the
\L{}ojasiewicz inequality. For the classical version over the real
ground field, we refer the reader to~\cite[Thm.2.6.6]{BCR}.

\begin{proposition}\label{L-1}
Let $f,g: A \to K$ be two continuous $\mathcal{L}$-definable
functions on a closed (in the $K$-topology)
$\mathcal{L}$-definable subset $A$ of $R^{m}$. If
$$ \{ x \in A: g(x)=0 \} \subset \{ x \in A: f(x)=0 \}, $$
then there exist a positive integer $s$ and a continuous
$\mathcal{L}$-definable function $h$ on $A$ such that $f^{s}(x) =
h(x) \cdot g(x)$ for all $x \in A$.
\end{proposition}

\begin{proof}
It is easy to check that the set
$$ A_{\gamma} := \{ x \in A: \; v(f(x))=\gamma \} $$
is a closed $\mathcal{L}$-definable subset of $A$ for every
$\gamma \in \Gamma$. Hence and by Corollary~\ref{clo-th-cor-0} to
the closedness theorem, the set $g(A_{\gamma})$ is a closed
$\mathcal{L}$-definable subset of $K \setminus \{ 0 \}$, $\gamma
\in \Gamma$. The set $v(g(A_{\gamma}))$ is thus bounded from
above, i.e.\
$$ v(g(A_{\gamma})) < \alpha (\gamma) $$
for some $\alpha(\gamma) \in \Gamma$. By elimination of
$K$-quantifiers, the set
$$ \Lambda := \{ (v(f(x)),v(g(x))) \in \Gamma^{2}: \; x \in A \}
   \subset \{ (\gamma,\delta) \in \Gamma^{2}: \; \delta < \alpha(\gamma) \} $$
is a definable subset of $\Gamma^{2}$, and thus it is described by
a finite number of linear inequalities and congruences. Hence
$$ \Lambda \, \cap  \{ (\gamma,\delta) \in \Gamma^{2}: \; \gamma > \gamma_{0} \}
   \subset \{ (\gamma,\delta) \in \Gamma^{2}: \; \delta < s \cdot \gamma \} $$
for a positive integer $s$ and some $\gamma_{0} \in \Gamma$. We
thus get
$$ v(g(x)) < s \cdot v(f(x)) \ \ \text{if} \ \ x \in A, \, v(f(x))> \gamma_{0}, $$
whence
$$ v(g(x)) < v(f^{s}(x)) \ \ \text{if} \ \ x \in A, \, v(f(x))> \gamma_{0}. $$
Consequently, the quotient $f^{s}/g$ extends by zero through the
zero set of the denominator to a (unique) continuous
$\mathcal{L}$-definable function on $A$. This finishes the proof.
\end{proof}

The above theorem can be generalized as follows.

\begin{proposition}\label{L-2}
Let $U$ and $F$ be two $\mathcal{L}$-definable subsets of $K^{m}$,
suppose $U$ is open and $F$ closed in the $K$-topology and
consider two continuous $\mathcal{L}$-definable functions  $f,g: A
\to K$ on the locally closed subset $A := U \cap F$ of $K^{m}$. If
$$ \{ x \in A: g(x)=0 \} \subset \{ x \in A: f(x)=0 \}, $$
then there exist a positive integer $s$ and a continuous
$\mathcal{L}$-definable function $h$ on $A$ such that $f^{s}(x) =
h(x) \cdot g(x)$ for all $x \in A$.
\end{proposition}

\begin{proof}
We shall adapt the foregoing arguments. Since the set $U$ is open,
its complement $V:= K^{m} \setminus U$ is closed in $K^{m}$ and
$A$ is the following union of open and closed subsets of $K^{m}$
and of $\mathbb{P}^{m}(K)$:
$$ X_{\beta} := \{ x \in K^{m}: \; v(x_{1}),\ldots,v(x_{m}) \geq
   -\beta, $$
$$  v(x-y) \leq \beta \ \ \ \text{for all} \ \ y \in V \}, $$
where $\beta \in \Gamma$, $\beta \geq 0$. As before, we see that
the sets
$$ A_{\beta,\gamma} := \{ x \in X_{\beta}: \; v(f(x))=\gamma \} \ \
   \text{with} \ \beta, \gamma \in \Gamma $$
are closed $\mathcal{L}$-definable subsets of $\mathbb{P}^{m}(K)$,
and next that the sets $g(A_{\beta,\gamma})$ are closed
$\mathcal{L}$-definable subsets of $K \setminus \{ 0 \}$ for all
$\beta,\gamma \in \Gamma$. Likewise, we get
$$ \Lambda := \{ (\beta,v(f(x)),v(g(x))) \in \Gamma^{3}: \; x \in X_{\beta} \}
   \subset $$
$$ \subset \{ (\beta,\gamma,\delta) \in \Gamma^{3}: \; \delta <
   \alpha(\beta,\gamma) \} $$
for some $\alpha(\beta,\gamma) \in\Gamma$. Consequently, since
$\Lambda$ is a definable subset of $\Gamma^{3}$, there exist a
positive integer $s$ and elements $\gamma_{0}(\beta) \in \Gamma$
such that
$$ \Lambda \cap \{ (\beta,\gamma,\delta) \in \Gamma^{3}: \;
   \gamma > \gamma_{0}(\beta) \} \subset \{ (\beta,\gamma,\delta) \in \Gamma^{3}:
   \; \delta < s \cdot \gamma \}. $$
Since $A$ is the union of the sets $X_{\beta}$, it is not
difficult to check that the quotient $f^{s}/g$ extends by zero
through the zero set of the denominator to a (unique) continuous
$\mathcal{L}$-definable function on $A$, which is the desired
result.
\end{proof}
%\vspace{2ex}

\section{Extending continuous hereditarily rational functions}
We first recall an elementary lemma from~\cite[Lemma~15]{K-N}.

%\vspace{2ex}

\begin{lemma}\label{zero-lemma}
If the ground field $K$ is not algebraically closed, then there
are polynomials $G_{r}(x_{1},\dots, x_{r})$ in any number of variables
whose only zero on $K^r$ is $(0,\dots,0)$. In particular,
$$ G_{2}(x_{1},x_{2}) = x_{1}^{d} + a_{1}x_{1}^{d-1}x_{2} + \cdots
   + a_{d}x_{2}^{d} \, , $$
where
$$ t^{d} + a_{1}t^{d-1} + \cdots +a_{d} \in K[t] $$
is a polynomial with no roots in $K$. \hspace*{\fill} $\Box$
\end{lemma}

We keep further the assumption that $K$ is a Henselian rank one
valued field of equicharacteristic zero and, additionally, that it
is not algebraically closed. We have at our disposal the descent
property (Corollary~\ref{clo-th-cor-4}) and the \L{}ojasiewicz
inequality (Proposition~\ref{L-2}). Therefore, by adapting mutatis
mutandis its proof, we are able to carry over Proposition~11
from~\cite{K-N} on extending continuous hereditarily rational
functions (being its main extension result) to the case of such
non-archimedean fields. We first recall the definition. Given a
$K$-variety $Z$, we say that a continuous function $f: Z(K) \to K$
is {\it hereditarily rational} if every irreducible subvariety $Y
\subset Z$ has a Zariski dense open subvariety $Y^{0} \subset Z$
such that $f|_{Y^{0}(K)}$ is regular.

%\vspace{2ex}

\begin{theorem}\label{ext-th} Let $X$ be a smooth $K$-variety and $W \subset Z\subset
X$ closed subvarieties. Let $f$ be a continuous hereditarily
rational function on $Z(K)$ that is regular at all $K$-points of
$Z(K) \setminus W(K)$. Then $f$ extends to a continuous
hereditarily rational function $F$ on $X(K)$ that is regular at
all $K$-points of $X(K) \setminus W(K)$.
\end{theorem}

\begin{remark}
The corresponding theorem for differentiable hereditarily rational
functions remains an open problem as yet
(cf.~Remark~\ref{rem-regulous} and the discussion afterwards).
\end{remark}

{\em Sketch of the Proof.} We shall keep the notation from the
paper~\cite{K-N}. The main modification of the proof in comparison
with that paper is the definition of the functions $G$ and
$F_{2n}$ which improve the rational function $P/Q$. Now we need
the following corrections:
$$ G := \frac{P}{Q} \cdot \frac{Q^{d}}{G_{2}(Q,H)} $$
and
$$ F_{dn} := G \cdot \frac{Q^{dn}}{G_{2}(Q^{n},H)} = \frac{P}{Q}
   \cdot\frac{Q^{d}}{G_{2}(Q,H)} \cdot
   \frac{Q^{dn}}{G_{2}(Q^{n},H)}, $$
where the positive integer $d$ and the polynomial $G_{2}$ are
taken from Lemma~\ref{zero-lemma}. It is clear that
the restriction of $F_{dn}$ to $Z \setminus W$ equals $f_{2}$, and
thus Theorem~\ref{ext-th} will be proven once we show that the
rational function $F_{dn}$ restricts to a continuous function
$\Phi_{dn}$ on $X(K)$ for $n \gg 1$.

\vspace{1ex}

We work on the variety $\pi: X_{1}(K) \longrightarrow X(K)$
obtained by blowing up the ideal $(PQ^{d-1}, G_{2}(Q,H))$; let
$E:= \pi^{-1}(W)$ be its exceptional divisor. Equivalently,
$X_{1}(K)$ is the Zariski closure of the graph of $G$ in $X(K)
\times \mathbb{P}^{1}(K)$. Two open charts are considered:

\vspace{1ex}

$\bullet$ \ a Zariski open neighbourhood $U^{*}$ of the closure
(in the $K$-topology) $Z^{*}$ of $\pi^{-1}(Z(K) \setminus W(K))$;

$\bullet$ \ an open (in the $K$-topology) set $V^{*} := X_{1}(K)
\setminus Z^{*}$, which is an $\mathcal{L}$-definable subset of
$X_{1}(K)$.

\vspace{1ex}

Via the descent property (Corollary~\ref{clo-th-cor-4}), it
suffices to show that the rational function $F_{dn} \circ \pi$
(with $n \gg 1$) extends to a continuous function on $X_{1}(K)$
that vanishes on $E(K)$.

\vspace{1ex}

The subtlest analysis is on the latter chart, on which $F_{dn}
\circ \pi$ can be written in the form
$$ F_{dn} \circ \pi = (P \circ \pi) \cdot \left(
   \frac{Q^{dn-1}}{H^{d}} \circ \pi \right) \cdot
   \left( \frac{Q^{d}}{G_{2}(Q,H)} \circ \pi \right) \cdot
   \left( \frac{H^{d}}{G_{2}(Q^{n},H)} \circ \pi \right). $$
Note that on $V^*$ the function $H \circ \pi$ vanishes only along
$E(K)$ and  $Q \circ \pi$ vanishes along $E(K)$ too. Therefore we
can apply Proposition~\ref{L-2} (\L{}ojasiewicz inequality) to the
numerator and denominator of the first factor, which are regular
functions on the chart $V^{*}$, to immediately deduce that the
first factor extends to a continuous rational function on $V^{*}$
that vanishes along $E(K) \cap V^{*}$ for $n \gg 1$.

\vspace{1ex}

What still remains to prove (cf.~\cite{K-N}) is that the factors
$$ \frac{Q^{dn}}{G_{2}(Q^{n},H)}, \ \ \frac{Q^{d}}{G_{2}(Q,H)} \ \
   \text{and} \ \ \frac{H^{d}}{G_{2}(Q^{n},H)} $$
are regular functions off $W(K)$ whose valuations are bounded from
below. But this follows immediately from an auxiliary lemma:

%\vspace{2ex}

\begin{lemma}\label{aux-lemma}
Let $g$ be the polynomial from the proof of
Lemma~\ref{zero-lemma}. Then the set of values
$$ v\left(\frac{t^{d}}{g(t)} \right) \in \Gamma, \ \ \ t \in K, $$
is bounded from below.
\end{lemma}

%\vspace{2ex}

In order to prove this lemma, observe that
$$ v\left( \frac{t^{d}}{g(t)} \right) = v\left( 1+ \frac{a_{1}}{t} + \cdots +
   \frac{a_{d}}{t^{d}} \right) $$
for $t \in K$, $ t \neq 0$. Hence the values under study are zero
if $i v(t) < v(a_{i})$ for all $i=1,\ldots,d$. Therefore, we are
reduced to analysing the case where
$$ v(t) \geq k := \min \, \left\{ \frac{v(a_{i})}{i}:
   \; i=1,\ldots,d \right\}. $$
Denote by $\Gamma$ the value group of $v$. Thus we must show that
the set of values $v(g(t)) \in \Gamma$ when $v(t) \geq k$ is
bounded from above.

\vspace{1ex}

Take elements $a,b \in R$, $a,b \neq 0$, such that $a a_{i} \in R$
for all $i=1,\ldots,d$, and $bt \in R$ whenever $v(t) \geq k$.
Then
$$ h(abt) := (ab)^{d}g(t) = $$
$$ = (abt)^{d} + aba_{1}(abt)^{d-1} + (ab)^{2}a_{2}(abt)^{d-2}
   + \cdots + (ab)^{d}a_{d} $$
is a monic polynomial with coefficients from $R$ which has no
roots in $K$. Clearly, it is sufficient to show that the set of
values $v(h(t)) \in \Gamma$ when $t \in R$ is bounded from above.

\vspace{1ex}

Consider a splitting field $\widetilde{K}=K(u_{1},\ldots,u_{d})$
of the polynomial $h$, where $u_{1},\ldots,u_{d}$ are the roots of
$h$. Let $\widetilde{v}$ be a (unique) extension to
$\widetilde{K}$ of the valuation $v$, $\widetilde{R}$ be its
valuation ring and $\widetilde{\Gamma} \supset \Gamma$
its value group (see e.g.\ \cite[Chap.~VI, \S~11]{Zar-Sam}
for valuations of algebraic field extensions).
Then
$$ u_{1},\ldots,u_{d} \in \widetilde{R} \setminus R \ \ \mbox{
   and } \ \ h(t) = \prod_{i=1}^{d} \, (t - u_{i}). $$
Since $R$ is a closed subring of $\widetilde{R}$ by
Lemma~\ref{closed-lemma}, there exists an $l \in
\widetilde{\Gamma}$ such that $\widetilde{v}(t-u_{i}) \leq l$ for
all $i=1,\ldots,d$ and $t \in R$. Hence $v(h(t)) \leq d l$ for all
$t \in R$, and thus the lemma follows. \hspace*{\fill} $\Box$

\vspace{2ex}

In this fashion, we have demonstrated how to adapt the proof of
Proposition~11 from~\cite{K-N} to the case of Henselian rank one
valued fields of equicharacteristic zero. Note that the proofs of
all remaining results from that paper work over general
topological fields with a density property introduced in~\cite[Section~3]{K-N}
and recalled below. A
topological field $K$ satisfies the {\it density property}
if one of the following equivalent conditions
holds:
\begin{enumerate}
\item If $X$ is a smooth, irreducible $K$-variety
 and $\emptyset \neq U\subset X$ is a Zariski open subset,
then $U(K)$ is dense in $X(K)$ in the $K$-topology. \item If $C$
is a smooth, irreducible $K$-curve
 and $\emptyset \neq C^{0} \subset C$ is a Zariski open subset,
then $C^{0}(K)$ is dense in $C(K)$ in the $K$-topology. \item If
$C$ is a smooth, irreducible $K$-curve, then $C(K)$ has no
isolated points.
\end{enumerate}
The examples of such fields are, in particular, all Henselian rank
one valued fields.

\section{Regulous functions and sets}
In these last three sections, we shall carry the theory of
regulous functions over the real ground field $\mathbb{R}$,
developed by Fichou--Huisman--Mangolte--Monnier~\cite{FHMM}, over
to non-archimedean algebraic geometry over Henselian rank one
valued fields $K$ of equicharacteristic zero. We assume that the
ground field $K$ is not algebraically closed. (Otherwise, the
notion of a regulous function on a normal variety coincides with
that of a regular function and, in general, the study of
continuous rational functions leads to the concept of {\it
seminormality} and {\it seminormalization}; cf.~\cite{A-B,A-N} or
\cite[Section~10.2]{Kol-1} for a recent treatment.) Every such
field enjoys the density property. The $K$-points $X(K)$ of any
algebraic $K$-variety $X$ inherit from $K$ a topology, called the
$K$-topology.

\vspace{1ex}

In this section, we deal with the ground fields $K$ with the
density property. Observe first that if $f$ is a rational function
on an affine $K$-variety $X$ which is regular on a Zariski open
subset $U$, then there exist two regular functions $p,q$ on $X$
such that
$$ f = \frac{p}{q} \ \ \ \text{and} \ \ \ q(x) \neq 0 \ \
   \text{for all} \ \ x \in U(K). $$
When $X \subset K\mathbb{A}^{n}$, then $p,q$ can be polynomial
functions. For every rational function $f$ on $X$, there is a
largest Zariski open subset of $X$ on which $f$ is regular, called
the regular locus of $f$ and denoted by $\text{dom}\, (f)$.
Further, assume that $Z$ is a closed subvariety of a $K$-variety
$X$. Then every rational function $f$ on $Z$ that is regular on
$Z(K)$ extends to a rational function $F$ on $X$ that is regular
on $X(K)$. Both the results can be deduced via
Lemma~\ref{zero-lemma} (cf.~\cite{K-N}, the proof of Lemma~15 on
extending regular functions).

\vspace{1ex}

Suppose now that $X$ is a smooth affine $K$-variety or, at least,
an affine $K$-variety that is smooth at all $K$-points $X(K)$. We
say that a function $f$ on $X(K)$ is $k$-regulous, $k \in
\mathbb{N} \cup \{ \infty \}$, if it is of class $\mathcal{C}^{k}$
and there is a Zariski dense open subset $U$ of $X$ such that the
restriction of $f$ to $U(K)$ is a regular function. A function $f$
on $X(K)$ is called regulous if it is $0$-regulous. Denote by
$\mathcal{R}^{k}(X(K))$ the ring of $k$-regulous functions on
$X(K)$. The ring $\mathcal{R}^{\infty}(X(K))$ of $\infty$-regulous
functions on $X(K)$ coincides with the ring $\mathcal{O}(X(K))$ of
regular functions on $X(K)$. This follows easily from the faithful
flatness of the formal power series ring $K[[x_{1},\ldots,x_{n}]]$
over the local ring of regular function germs at $0 \in K^{n}$.
Therefore we shall restrict ourselves to the case $k \in
\mathbb{N}$.

\vspace{1ex}

When $K$ is a Henselian rank one valued field of
equicharacteristic zero, transformation to a normal crossing by
blowing up along smooth centers and the descent property
(Corollary~\ref{clo-th-cor-4}) enable the following
characterization:

\vspace{1ex}

\begin{em}
Given a smooth algebraic $K$-variety $X$, a function
$$ f: X(K) \to K $$
is regulous iff there exists a finite composite $\sigma:
\widetilde{X} \to X$ of blow-ups with smooth centers such that the
pull-back $f^{\sigma} := f \circ \sigma$ is a regular function on
$\widetilde{X}(K)$.
\end{em}

\vspace{1ex}

We say that a subset $V$ of $K^{n}$ is $k$-regulous closed if it
is the zero set of a family $E \subset
\mathcal{R}^{k}(\mathbb{K}^{n})$ of $k$-regulous functions:
$$ V = \mathcal{Z}(E) := \{ x \in K^{n}: \ f(x)=0 \ \ \text{for
   all} \ f \in E \}. $$
A subset $U$ of $K^{n}$ is called $k$-regulous open if its
complement $\mathbb{K}^{n} \setminus U$ is $k$-regulous closed.
The family of $k$-regulous open subsets of $K^{n}$ is a topology
on $K^{n}$, called the $k$-regulous topology on $K^{n}$.

\vspace{1ex}

If $f \neq 0$ is a $k$-regulous function on $K^{n}$ with regular
locus
$$ U = \text{dom}\, (f) \subset K^{n}, $$
then
$$ f = \frac{p}{q} \ \ \ \text{where} \ \ p,q \in
   K[x_{1},\ldots,x_{n}], \ \mathcal{Z}(q) = K^{n} \setminus U(K), $$
and $p,q$ are coprime polynomials. Clearly, $\mathcal{Z}(q)
\subset \mathcal{Z}(p)$ and it follows, by passage to the
algebraic closure of $K$, that the zero set $\mathcal{Z}(q)$ is of
codimension $\geq 2$ in $K^{n}$. Thus the complement $K^{n}
\setminus \text{dom}\, (f)$ is of codimension $ \geq 2$ in
$K^{n}$. Consequently, every $k$-regulous function on $K$ is
regular and every $k$-regulous function on $K^{2}$ is regular at
all but finitely many points.

\vspace{1ex}

We now recall some results about algebraic varieties over
arbitrary fields $F$. Let $V$ be an affine $F$-variety. We are
interested in the set $V(F)$ of its $F$-points. Therefore, from
now on, we shall (and may) assume that $V(F)$ is Zariski dense in
$V$. Then the regular locus ${\rm Reg}\, (V)$ of $V$ is a
non-empty, Zariski open subset of $V$ and, moreover, its trace on
the set $V(F)$ is non-empty; cf.~\cite{Ku}, Chap.~VI,
Corollary~1.17 to the Jacobian criterion for regular local rings
and the remark preceding it. If the ground field $F$ is not
algebraically closed, then the trace ${\rm Reg}\, (V) \cap V(F)$
is smooth and affine, because it follows immediately from
Lemma~\ref{zero-lemma} that every algebraic subset of $F^{n}$ is
the zero set of one polynomial. Summing up, we see that, for every
affine $F$-subvariety $V$ of $F\mathbb{A}^{n}$, the set $V(F)
\subset F^{n}$ of its $F$-points is a finite (disjoint) union of
smooth, affine, Zariski locally closed subsets of $F^{n}$ of pure
dimension.

\vspace{1ex}

We call a subset $E$ of $K^{n}$ constructible if it is a (finite)
Boolean combination of Zariski closed subsets of $K^{n}$. Every
such set $E$ is, of course, a finite union of Zariski locally
closed subsets. Further, in view of the foregoing discussion, $E$
is a finite union of smooth, affine, Zariski locally closed
subsets of pure dimension with irreducible Zariski closure. Thus
it follows immediately from the density property that every closed
(in the $K$-topology) constructible subset $E$ of $K^{n}$ is a
finite union of a unique irredundant family $\Sigma(E)$ of
constructible subsets each of which is the regular locus ${\rm
Reg}\, (V) \cap K^{n}$ of an irreducible affine $K$-subvariety $V$
of $K\mathbb{A}^{n}$; obviously, ${\rm Reg}\, (V) \cap K^{n}$ is a
smooth, Zariski locally closed subset of pure dimension $\dim V$.

\vspace{1ex}

Below we introduce the constructible topology on $K^{n}$. We shall
see in the next section that the $k$-regulous topology coincides
with the constructible topology for all $k \in \mathbb{N}$.

\begin{proposition}\label{construct-top}
If $K$ is a topological field with the density property, then the
family of all closed (in the $K$-topology) constructible subsets
of $K^{n}$ is the family of closed sets for a topology, called the
constructible topology on $K^{n}$. Furthermore, this topology is
noetherian, i.e.\ every descending sequence of closed
constructible subsets of $K^{n}$ stabilizes.
\end{proposition}

\begin{proof} Clearly, it suffices to prove only the last assertion.
We shall follow the reasoning from our paper~\cite{Now0} which
showed that the quasianalytic topology is noetherian. For any
closed (in the $K$-topology) constructible subset $E$ of $K^{n}$,
let $\mu_{i}(E)$ be the number of elements from the family
$\Sigma(E)$, constructed above, of dimension $i$,
$i=0,1,\ldots,n$, and put
$$ \mu(E) = (\mu_{n}(E), \mu_{n-1}(E), \ldots, \mu_{0}(E))
   \in \mathbb{N}^{n+1}. $$
Consider now a descending sequence of closed constructible subsets
$$ K^{n} \supset E_{1} \supset E_{2} \supset E_{3} \supset \ldots \,. $$
It is easy to check that for any two closed (in the $K$-topology)
constructible subsets $D \subset E$, we have $\mu(D) \leq \mu(E)$
and, furthermore, \ $D=E$ \ iff \ $\mu(D) = \mu(E)$. Hence we get
the decreasing (in the lexicographic order) sequence of
multi-indices
$$ \mu(E_{1}) \geq \mu(E_{2}) \geq \mu(E_{3}) \geq \ldots \,, $$
which must stabilize for some $N \in \mathbb{N}$:
$$ \mu(E_{N}) = \mu(E_{N+1}) = \mu(E_{N+2}) = \ldots  $$
Then
$$ E_{N} = E_{N+1} = E_{N+2} = \ldots \,, $$
as desired.
\end{proof}

\begin{corollary}\label{correspond-prop}
Suppose $K$ is a field with the density property. Then there is a
one-to-one correspondence between the irreducible closed
constructible subsets $E$ of $K^{n}$ and the irreducible Zariski
closed subsets $V$ of $K^{n}$:
$$ \alpha: E \longmapsto \overline{E}^{Z} \ \ \ \text{and} \ \ \
   \beta: V \longmapsto \overline{\text{Reg}\, (V)}^{c},   $$
where $\overline{E}^{Z}$ stands for the Zariski closure of $E$ and
$\overline{A}^{c}$ for the closure of $A$ in the constructible
topology.
\end{corollary}

\begin{proof} We have $\alpha \circ \beta = \text{Id}$, because
$\overline{\text{Reg}\, (V)}^{Z} = V$ for every irreducible
Zariski closed subset $V$ of $K^{n}$. In view of the foregoing
discussion, the assignment $\beta$ is surjective. Therefore every
irreducible closed constructible subset $E$ of $K^{n}$ is of the
form $E = \beta(V) = \overline{\text{Reg}\, (V)}^{c}$ for an
irreducible Zariski closed subset $V$ of $K^{n}$. Hence
$$ (\beta \circ \alpha)(E) = (\beta \circ \alpha)(\beta(V)) =
   (\beta \circ \alpha \circ \beta)(V)
   = (\beta \circ \text{Id})(V) = \beta(V) = E, $$
and thus $\beta \circ \alpha = \text{Id}$, which finishes the
proof.
\end{proof}

Below we recall Proposition~8 from~\cite{K-N} which holds over any
topological fields with the density property.

\begin{proposition}\label{filtration}
Let $X$ be an algebraic $K$-variety and $f$ a rational function on
$X$ that is regular on $X^{0} \subset X$. Assume that
$f|_{X^{0}(K)}$ has a continuous extension $f^{c}: X(K)\to K$. Let
$Z\subset X$ be an irreducible subvariety that is not contained in
the singular locus of $X$. Then there is a Zariski dense open
subset $Z^{0}\subset Z$ such that $f^c|_{Z^{0}(K)}$ is a regular
function. \ \ \hspace*{\fill} $\Box$
\end{proposition}

We immediately obtain two corollaries:

\begin{corollary}\label{filtration-cor-1}
Let $X$ be an algebraic $K$-variety that is smooth at all
$K$-points $X(K)$ and $f$ a rational function on $X$ that is
regular on $X^{0} \subset X$. Assume that $f|_{X^{0}(K)}$ has a
continuous extension $f^{c}: X(K)\to K$. Then there is a sequence
of closed subvarieties
$$ \emptyset = X_{-1} \subset X_{0} \subset \cdots \subset X_{n}=X $$
such that for  $i=0,\dots, n$ the restriction of $f$ to $X_{i}(K)
\setminus X_{i-1}(K)$ is regular. Moreover, we can require that
each set $X_{i} \setminus X_{i-1}$ be smooth of pure dimension
$i$.  \hspace*{\fill} $\Box$
\end{corollary}

\begin{corollary}\label{filtration-cor-2}
If $f$ is a regulous function on $K^{n}$, then there is a sequence
of Zariski closed subsets
$$ \emptyset = E_{-1} \subset E_{0} \subset \cdots \subset E_{n}= K^{n} $$
such that for  $i=0,\dots, n$ the restriction of $f$ to $E_{i}
\setminus E_{i-1}$ is regular. Moreover, we can require that each
set $E_{i} \setminus E_{i-1}$ be smooth of pure dimension $i$.
\hspace*{\fill} $\Box$
\end{corollary}

\begin{remark}\label{filtr}
Given a finite number of regulous functions $f_{1},\ldots,f_{p}$,
there is a filtration
$$ \emptyset = E_{-1} \subset E_{0} \subset \cdots \subset E_{n}= K^{n} $$
as in the above corollary such that for $i=0,\dots, n$ the
restriction of each function $f_{j}$, $j=1,\ldots,p$, to $E_{i}
\setminus E_{i-1}$ is regular.
\end{remark}

Now, three further consequences of the above corollary will be
drawn. We say that a map
$$ f = (f_{1},\ldots,f_{p}): K^{n} \to K^{p} $$
is $k$-regulous if all its components $f_{1},\ldots,f_{p}$ are
$k$-regulous functions on $K^{n}$.

\begin{corollary}\label{filtration-cor-3}
If two maps
$$ g: K^{m} \to K^{n} \ \ \ \text{and} \ \ \ f: K^{n} \to K^{p} $$
are $k$-regulous, so is its composition $f \circ g$.
\end{corollary}

\begin{proof} Indeed, let $U$ be the common regular locus of the
components $g_{1},\ldots,g_{n}$ of the map $g$:
$$ U := \text{dom}\, (g_{1}) \cap \ldots \cap \text{dom}\, (g_{n})
   \subset \mathbb{R}^{m}. $$
Take a filtration
$$ \emptyset = E_{-1} \subset E_{0} \subset \cdots \subset E_{n}= K^{n} $$
for the functions $f_{1},\ldots,f_{p}$ described in
Remark~\ref{filtr}. Then $U$ is the following union of Zariski
locally closed subsets
$$ U = \bigcup_{i=0}^{n} \left( U \cap g^{-1}
   (E_{i} \setminus E_{i-1}) \right). $$
Clearly, one of these sets, say $U \cap g^{-1} (E_{i_{0}}
\setminus E_{i_{0}-1})$  must be a Zariski dense open subset of
$U$ and of $K^{m}$ too. Hence $f \circ g$ is a regular function on
$U \cap g^{-1} (E_{i_{0}} \setminus E_{i_{0}-1})$, which is the
required result.
\end{proof}

\begin{corollary}\label{filtration-cor-4}
The zero set $\mathcal{Z}(f)$ of a regulous function $f$ on
$K^{n}$ is a closed (in the $K$-topology) constructible subset of
$K^{n}$. \hspace*{\fill} $\Box$
\end{corollary}

\begin{corollary}\label{filtration-cor-5}
The zero set $\mathcal{Z}(f_{1},\ldots,f_{p})$ of finitely many
regulous functions $f_{1},\ldots,f_{p}$ on $K^{n}$ is a closed (in
the $K$-topology) constructible subset of $K^{n}$.
\end{corollary}
\begin{proof} This follows directly from
Corollary~\ref{filtration-cor-4} and Lemma~\ref{zero-lemma}.
\end{proof}

Hence and by Proposition~\ref{construct-top}, we immediately
obtain

\begin{proposition}\label{regul-top}
The $k$-regulous topology on $K^{n}$ is noetherian.
\hspace*{\fill} $\Box$
\end{proposition}

\begin{corollary}\label{regul-prop-cor-1}
Every $k$-regulous closed subset of $K^{n}$ is the zero set
$\mathcal{Z}\, (f)$ of a $k$-regulous function $f$ on $K^{n}$, and
thus is a closed (in the $K$-topology) constructible subset of
$K^{n}$. Hence every $k$-regulous open subset of $K^{n}$ is
 of the form
$$ \mathcal{D}\, (f) := K^{n} \setminus \mathcal{Z}\, (f) =
   \{ x \in K^{n}: \; f(x) \neq 0 \} $$
for a $k$-regulous function $f$ on $K^{n}$.  \hspace*{\fill}
$\Box$
\end{corollary}

Corollaries~\ref{regul-prop-cor-1} and~\ref{filtration-cor-3}
yield the following

\begin{corollary}\label{regul-prop-cor-2}
Every $k$-regulous map $f: K^{n} \to K^{m}$ is continuous in the
$k$-regulous topology.  \hspace*{\fill} $\Box$
\end{corollary}

\section{Regulous Nullstellensatz}
We further assume that the ground field $K$ is a
Henselian rank one valued field of equicharacteristic zero and
that $K$ is not algebraically closed. Throughout this section, $k$
will be a non-negative integer. We begin with the following consequence of the
\L{}ojasiewicz inequality (Proposition~\ref{L-2}).

\begin{proposition}\label{Loj-prop}
Let $f,g$ be rational functions on $K\mathbb{A}^{n}$ such that $f$
extends to a continuous function on $K^{n}$ and $g$ extends to a
continuous function on the set $\mathcal{D}\, (f)$. Then the
function $f^{s} g$ extends, for $s \gg 0$, by zero through the set
$\mathcal{Z}\, (f)$ to a continuous rational function on $K^{n}$.
\end{proposition}

\begin{proof} We can find a finite composite $\sigma: M \to K\mathbb{A}^{n}$
of blow-ups along smooth centers such that the pull-backs
$$ f^{\sigma} := f \circ \sigma \ \ \ \text{and} \ \ \ g^{\sigma}
   := g \circ \sigma $$
are regular functions at all $K$-points on
$$ M(K) \ \ \ \text{and} \ \ \ M(K) \setminus \sigma^{-1}(\mathcal{Z}\, (f)), $$
respectively. Then there are regular functions $p,q$ on $M$ such
that
$$ g^{\sigma} = \frac{p}{q} \ \ \ \text{and} \ \ \ \mathcal{Z}\,
  (q) := \{ y \in M(K): \; q(y) =0 \} \subset \mathcal{Z}\,
  (f^{\sigma}). $$
It follows immediately from Proposition~\ref{L-2} that the
rational function
$$ \frac{(f^{\sigma})^{s}}{q}\,, \ \ \  \text{for} \ \ s \gg 0, $$
extends by zero through the set $\mathcal{Z}\, (f^{\sigma})$ to a
continuous function on $M(K)$, whence so does the rational
function $(f^{\sigma})^{s} \cdot g^{\sigma}$. By the descent
property (Corollary~\ref{clo-th-cor-4}), the continuous function
$(f^{\sigma})^{s} \cdot g^{\sigma}$ descends to a continuous
function on $K^{n}$ that vanishes on $\mathcal{Z}\, (f)$. This is
the required result.
\end{proof}

The two corollaries stated below are counterparts of Lemmata~5.1
and~5.2 from~\cite{FHMM}, established over the real ground field
$\mathbb{R}$.

\begin{corollary}\label{Loj-prop-cor-1}
Let $f$ be a $k$-regulous function on $K^{n}$ and $g$ a
$k$-regulous function on the open subset $\mathcal{D}\, (f)$. Then
the function $f^{s} g$, for $s \gg 0$, extends by zero through the
zero set $\mathcal{Z}\, (f)$ to a $k$-regulous function on
$K^{n}$. Hence the ring of $k$-regulous functions on
$\mathcal{D}\, (f)$ is the localization
$\mathcal{R}^{k}(K^{n})_{f}$.
\end{corollary}

\begin{proof}
The case $k=0$ is just Proposition~\ref{Loj-prop}. Now take $s$
large enough so that the partial derivatives
$$ f^{s} \cdot \frac{\partial^{|\alpha|}g}{\partial x^{\alpha}} \ \ \
   \mbox{for} \ \ \alpha \in \mathbb{N}^{n}, \ |\alpha|=0,1,\ldots,k, $$
extend to continuous functions on $K^{n}$ vanishing on
$\mathcal{Z}\, (f)$. Then, by Leibniz's rule, $f^{s+k} g$ is a
$k$-regulous function on $K^{n}$ and $k$-flat on $\mathcal{Z}\,
(f)$, as desired.
\end{proof}

\begin{corollary}\label{Loj-prop-cor-2}
Let $U$ be a $k$-regulous open subset of $K^{n}$, $f$ a
$k$-regulous function on $U$ and $g$ a $k$-regulous function on
the open subset $\mathcal{D}\, (f) \subset U$. Then the function
$f^{s} g$ extends, for $s \gg 0$, by zero through the zero set
$\mathcal{Z}\, (f) \subset U$ to a $k$-regulous function on $U$.
\end{corollary}

\begin{proof} By Corollary~\ref{regul-prop-cor-1}, $U = \mathcal{D}(h)$ for a $k$-regulous
function on $K^{n}$. From the above corollary we get
$$ fh^{s} \in \mathcal{R}^{k}(K^{n}) \ \ \ \text{for} \ \ s \gg 0. $$
Hence
$$ g \in \mathcal{R}^{k}(K^{n})_{fh^{s}} $$
for integers $s$ large enough, and thus the conclusion follows.
\end{proof}

Now we can readily pass to a regulous version of Nullstellensatz,
whose proof relies on Corollary~\ref{Loj-prop-cor-1} and the fact
that the $k$-regulous topology is noetherian.

\begin{theorem}\label{nullstellen}
If $I$ is an ideal in the ring $\mathcal{R}^{k}(K^{n})$ of
$k$-regulous functions on $K^{n}$, then
$$ {\rm Rad}\, (I) = \mathcal{I}\,(\mathcal{Z}\,(I)), $$
where
$$ \mathcal{I}\, (E) := \{ f \in \mathcal{R}^{k}(K^{n}): \; f(x)=0
   \ \ \text{for all } \ x \in E \} $$
for a subset $E$ of $K^{n}$.
\end{theorem}

\begin{proof} The inclusion ${\rm Rad}\, (I) \subset \mathcal{I}\,(\mathcal{Z}\,(I))$
is obvious. For the converse one, apply
Corollary~\ref{regul-prop-cor-1} which says that there is a
function $g \in I$ such that $\mathcal{Z}\, (I) = \mathcal{Z}\,
(g)$. Then $\mathcal{Z}\, (g) \subset \mathcal{Z}\, (f)$ for any
$f \in \mathcal{I}\,(\mathcal{Z}\,(I))$, and thus the function
$1/g$ is $k$-regulous on the set $\mathcal{D}\, (f)$. By
Corollary~\ref{Loj-prop-cor-1}, we get
$$ \frac{f^{s}}{g} \in \mathcal{R}^{k}(K^{n}) $$
for $s \gg 0$ large enough. Hence
$$ f^{s} \in g \cdot \mathcal{R}^{k}(K^{n}) \subset I, $$
concluding the proof.
\end{proof}

\begin{corollary}\label{nullstellen-cor-1}
There is a one-to-one correspondence between the radical ideals of
the ring $\mathcal{R}^{k}(K^{n})$ and the $k$-regulous closed
subsets of $K^{n}$. Consequently, the prime ideals of
$\mathcal{R}^{k}(K^{n})$ correspond to the irreducible
$k$-regulous closed subsets of $K^{n}$, and the maximal ideals
$\mathfrak{m}$ of $\mathcal{R}^{k}(K^{n})$ correspond to the
points $x$ of $K^{n}$ so that we get the bijection
$$ K^{n} \ni x \longrightarrow \mathfrak{m}_{x} := \{ f \in \mathcal{R}^{k}(K^{n}):
   f(x)=0 \} \in \text{\rm Max} \left( \mathcal{R}^{k}(K^{n}) \right). $$
The resulting embedding
$$ \iota: K^{n} \ni x \longrightarrow \mathfrak{m}_{x} \in \text{\rm Spec}
   \left( \mathcal{R}^{k}(K^{n}) \right) $$
is continuous in the $k$-regulous and Zariski topologies.
Furthermore, $\iota$ induces a canonical one-to-one correspondence
between the $k$-regulous closed subsets of $K^{n}$ and the Zariski
closed subsets of $\text{\rm Spec} \left( \mathcal{R}^{k}(K^{n})
\right)$. More precisely, for every $k$-regulous closed subset $V$
of $K^{n}$ there is a unique Zariski closed subset $\widetilde{V}$
of $\mathcal{R}^{k}(K^{n})$ such that $V =
\iota^{-1}(\widetilde{V})$; actually $\widetilde{V}$ is the
Zariski closure of the image $\iota(V)$.
\end{corollary}

\begin{proof}
The embedding $\iota$ is continuous by the very definition of the
Zariski topology. The last assertion follows immediately from the
Nullstellensatz and the fact that the closed subsets of $\text{\rm
Spec} \left( \mathcal{R}^{k}(K^{n}) \right)$ are precisely of the
form
$$ \{ \mathfrak{p} \in \text{\rm Spec}
   \left( \mathcal{R}^{k}(K^{n}) \right): \ \mathfrak{p} \supset I \}, $$
where $I$ runs over all radical ideals of $\text{\rm Spec} \left(
\mathcal{R}^{k}(K^{n}) \right)$.
\end{proof}

The above corollary along with Proposition~\ref{regul-top} and
Corollary~\ref{regul-prop-cor-1} yield immediately

\begin{corollary}\label{nullstellen-cor-1.5}
With the above notation, the space $\text{\rm Spec} \left(
\mathcal{R}^{k}(K^{n}) \right)$ with the Zariski topology is
noetherian, and the embedding $\iota$ induces a one-to-one
correspondence between the $k$-regulous open subsets of $K^{n}$
and the subsets of $\text{\rm Spec} \left( \mathcal{R}^{k}(K^{n})
\right)$ open in the Zariski topology. In particular, every open
subset of $\text{\rm Spec} \left( \mathcal{R}^{k}(K^{n}) \right)$
is of the form
$$ \mathcal{U}(f) := \{ \mathfrak{p} \in \text{\rm Spec}
   \left( \mathcal{R}^{k}(K^{n}) \right): \ f \not \in \mathfrak{p} \}, \ \ \
   f \in \mathcal{R}^{k}(K^{n}), $$
corresponding to the subset $\mathcal{D}(f)$ of $K^{n}$.
\hspace*{\fill} $\Box$
\end{corollary}

\begin{remark}\label{remark-non-noe}
As demonstrated in~\cite{FHMM} (see also~\cite[Ex.~6.11]{Kur}),
the ring $\mathcal{R}^{k}(K^{n})$ is not noetherian for all $k, n
\in \mathbb{N}$, $n \geq 2$.
\end{remark}

From Theorem~\ref{nullstellen} and
Corollary~\ref{regul-prop-cor-1}, we immediately obtain

\begin{corollary}\label{nullstellen-cor-2}
Every radical ideal of $\mathcal{R}^{k}(K^{n})$ is the radical of
a principal ideal of $\mathcal{R}^{k}(K^{n})$. \hspace*{\fill}
$\Box$
\end{corollary}

Finally, we return to the the comparison of the regulous and
constructible topologies. Below we state the non-archimedean
version of \cite[Theorem~6.4]{FHMM} by
Fichou--Huisman--Mangolte--Monnier, which says that those
topologies coincide in the real algebraic geometry. The proof
relies on their Lemmata~5.1 and~5.2, and it can be repeated
verbatim in the case of the ground fields $K$ studied in our paper
by means of Corollaries~\ref{Loj-prop-cor-1}
and~\ref{Loj-prop-cor-2}.

\begin{proposition}\label{compari-th}
The $k$-regulous closed subsets of $K^{n}$ are precisely the
closed (in the $K$-topology) constructible subsets of $K^{n}$.
\hspace*{\fill} $\Box$
\end{proposition}

The above theorem along with Corollary~\ref{correspond-prop} and
Corollary~\ref{nullstellen-cor-1} yield the following

\begin{corollary}\label{compari-th-cor-1}
There are one-to-one correspondences between the prime ideals of
the ring $\mathcal{R}^{k}(K^{n})$, the irreducible closed
constructible subsets of $K^{n}$ and the irreducible Zariski
closed subsets of $K^{n}$. \hspace*{\fill} $\Box$
\end{corollary}

\begin{corollary}\label{compari-th-cor-2}
The dimension of the topological space $K^{n}$ with the regulous
topology and the Krull dimension of the ring
$\mathcal{R}^{k}(K^{n})$ is $n$.  \hspace*{\fill} $\Box$
\end{corollary}

\section{Quasi-coherent regulous sheaves}
The concepts of quasi-coherent $k$-regulous sheaves on $K^{n}$ and
$k$-regulous affine varieties, $k \in \mathbb{N} \cup \{ \infty
\}$, can be introduced over valued fields studied in this paper,
similarly as by Fichou--Huisman--Mangolte--Monnier~\cite{FHMM}
over the real ground field $\mathbb{R}$. Also, the majority of
their results concerning these concepts carry over to the
non-archimedean geometry with similar proofs. For the sake of
completeness, we provide an exposition of the theory of
quasi-coherent regulous sheaves. Here we shall deal only with
$k$-regulous functions with a non-negative integer $k$, because
for $k = \infty$ we encounter the classical case of regular
functions and quasi-coherent algebraic sheaves.

\vspace{1ex}

Consider an affine scheme $Y = \text{\rm Spec}\, (A)$ with
structure sheaf $\mathcal{O}_{Y}$. Any $A$-module $M$ determines a
quasi-coherent sheaf $\widetilde{M}$ on $Y$
(cf.~\cite[Chap.~II]{Ha}). The functor $M \longmapsto
\widetilde{M}$ gives an equivalence of categories between the
category of $A$-modules and the category of quasi-coherent
$\mathcal{O}_{Y}$-modules. Its inverse is the global sections
functor
$$ \mathcal{F} \longmapsto H^{0}(Y,\mathcal{F}) $$
(cf.~\cite[Chap.~II, Corollary~5.5]{Ha}).

\vspace{1ex}

Denote by $\widetilde{\mathcal{R}}^{k}$ the structure sheaf of the
affine scheme $\text{\rm Spec} \left( \mathcal{R}^{k}(K^{n})
\right)$ and by $\mathcal{R}^{k}$ the sheaf of $k$-regulous
function germs (in the $k$-regulous topology equal to the
constructible topology) on $K^{n}$. It follows directly from
Corollaries~\ref{nullstellen-cor-1} and~\ref{Loj-prop-cor-1} that
the restriction $\iota^{-1} \widetilde{\mathcal{R}}^{k}$ of
$\widetilde{\mathcal{R}}^{k}$ to $K^{n}$ coincides with the sheaf
$\mathcal{R}^{k}$; conversely, \ $\iota_{*} \, \mathcal{R}^{k} =
\widetilde{\mathcal{R}}^{k}$.

\vspace{1ex}

By a $k$-regulous sheaf $\mathcal{F}$ we mean a sheaf of
$\mathcal{R}^{k}$-modules. Again, it follows immediately from
Corollaries~\ref{nullstellen-cor-1} and~\ref{nullstellen-cor-1.5}
that the functor $\iota^{-1}$ of restriction to $K^{n}$ gives an
equivalence of categories between
$\widetilde{\mathcal{R}}^{k}$-modules and
$\mathcal{R}^{k}$-modules. Its inverse is the direct image functor
$\iota_{*}$.

\vspace{1ex}

We say that $\mathcal{F}$ is a quasi-coherent $k$-regulous sheaf
on $K^{n}$ if it is the restriction to $K^{n}$ of a quasi-coherent
$\widetilde{\mathcal{R}}^{k}$-module. Thus the functor
$\iota^{-1}$ induces an equivalence of categories between
quasi-coherent $\widetilde{\mathcal{R}}^{k}$-modules and
quasi-coherent $\mathcal{R}^{k}$-modules, whose inverse is the
direct image functor $\iota_{*}$. For any
$\mathcal{R}^{k}(K^{n})$-module $M$, we shall denote by
$\widetilde{M}$ both the associated sheaf on $\text{\rm Spec}
\left( \mathcal{R}^{k}(K^{n}) \right)$ and its restriction to
$K^{n}$. This abuse of notation does not lead to confusion. We
thus obtain the following version of Cartan's Theorem~A:

\begin{theorem}\label{Cartan-A}
The functor $M \longmapsto \widetilde{M}$ gives an equivalence of
categories between the category of
$\mathcal{R}^{k}(K^{n})$-modules and the category of
quasi-coherent $\mathcal{R}^{k}$-modules. Its inverse is the
global sections functor
$$ \mathcal{F} \longmapsto H^{0}(K^{n},\mathcal{F}). $$
In particular, every quasi-coherent sheaf $\mathcal{F}$ is
generated by its global sections $H^{0}(K^{n},\mathcal{F})$.
\hspace*{\fill} $\Box$
\end{theorem}

The regulous version of Cartan's Theorem~B, stated below, follows
directly from the version for affine (not necessarily noetherian
in view of Remark~\ref{remark-non-noe}) schemes
(cf.~\cite[Theorem~1.3.1]{Gro}) via the discussed equivalence of
categories (being the functor $\iota^{-1}$ of restriction to
$K^{n}$).

\begin{theorem}\label{Cartan-B}
If $\mathcal{F}$ is a quasi-coherent $k$-regulous sheaf on
$K^{n}$, then
$$ H^{i}(K^{n},\mathcal{F}) = 0  \ \ \ \text{for all} \ \ i > 0. $$
\hspace*{\fill} $\Box$
\end{theorem}

\begin{corollary}\label{B-cor-1}
The global sections functor
$$ \mathcal{F} \longmapsto H^{0}(K^{n},\mathcal{F}) $$
on the category of quasi-coherent $k$-regulous sheaves on $K^{n}$
is exact. \hspace*{\fill} $\Box$
\end{corollary}

Let $V$ be a $k$-regulous closed subset of $K^{n}$ and
$\mathfrak{I}(V)$ the sheaf of those $k$-regulous function germs
on $K^{n}$ that vanish on $V$. It is a quasi-coherent sheaf of
ideals of $\mathcal{R}^{k}$, the sheaf
$\mathcal{R}^{k}/\mathfrak{I}(V)$ has support $V$ and is generated
by its global sections (Theorem A); moreover
$$ H^{0} \left( K^{n},\mathcal{R}^{k}/\mathfrak{I}(V) \right) =
   H^{0} \left( K^{n},\mathcal{R}^{k} \right)/H^{0} \left(
   K^{n},\mathfrak{I}(V) \right) $$
(Theorem B). The subset $V$ inherits the $k$-regulous topology
from $K^{n}$ and constitutes, together with the restriction
$\mathcal{R}^{k}_{V}$ of the sheaf
$\mathcal{R}^{k}/\mathfrak{I}(V)$ to $V$, a locally ringed space
of $K$-algebras, called an affine $k$-regulous subvariety of
$K^{n}$. More generally, by an affine $k$-regulous variety we mean
any locally ringed space of $K$-algebras that is isomorphic to an
affine $k$-regulous subvariety of $K^{n}$ for some $n \in
\mathbb{N}$.

\vspace{1ex}

We can define in the ordinary fashion the category of
quasi-coherent $\mathcal{R}^{k}_{V}$-modules. Each such module
extends trivially by zero to a quasi-coherent
$\mathcal{R}^{k}$-module on $K^{n}$. The sections
$\mathcal{R}^{k}_{V}(V)$ of the structure sheaf
$\mathcal{R}^{k}_{V}$ are called $k$-regulous functions on $V$. It
follows from Cartan's Theorem B that each $k$-regulous function on
$V$ is the restriction to $V$ of a $k$-regulous function on
$K^{n}$. Hence we immediately obtain the following two results.

\begin{proposition}
Let $W$ and $V$ be two affine $k$-regulous subvarieties of $K^{m}$
and $K^{n}$, respectively. Then the following three conditions are
equivalent:

1) $f: W \to V$ is a morphism of locally ringed spaces;

2) $f= (f_{1},\ldots,f_{n}): W \to K^{n}$ where
$f_{1},\ldots,f_{n}$ are $k$-regulous functions on $W$ such that
$f(W) \subset V$;

3) $f$ extends to a $k$-regulous map $K^{m} \to K^{n}$.
\hspace*{\fill} $\Box$
\end{proposition}

We then call $f:W \to V$ a $k$-regulous map.

\begin{corollary}
Let $W$, $V$ and $X$ be affine $k$-regulous subvarieties of
$K^{m}$, $K^{n}$ and $K^{p}$, respectively. If two maps
$$ g: W \to V \ \ \ \text{and} \ \ \ f: V \to X $$
are $k$-regulous, so is its composition $f \circ g$.
\hspace*{\fill} $\Box$
\end{corollary}

It is clear that Cartan's theorems remain valid for quasi-coherent
$k$-regulous sheaves on affine $k$-regulous varieties $V$.

\begin{corollary}\label{Cartan-A1}
The functor $M \longmapsto \widetilde{M}$ gives an equivalence of
categories between the category of
$\mathcal{R}^{k}_{V}(V)$-modules and the category of
quasi-coherent $\mathcal{R}^{k}_{V}$-modules. Its inverse is the
global sections functor
$$ \mathcal{F} \longmapsto H^{0}(V,\mathcal{F}). $$
In particular, every quasi-coherent sheaf $\mathcal{F}$ is
generated by its global sections $H^{0}(V,\mathcal{F})$.
\hspace*{\fill} $\Box$
\end{corollary}

\begin{corollary}\label{Cartan-B1}
If $\mathcal{F}$ is a quasi-coherent $k$-regulous sheaf on $V$,
then
$$ H^{i}(V,\mathcal{F}) = 0  \ \ \ \text{for all} \ \ i > 0. $$
\hspace*{\fill} $\Box$
\end{corollary}

\begin{corollary}\label{B1-cor-1}
The global sections functor
$$ \mathcal{F} \longmapsto H^{0}(V,\mathcal{F}) $$
on the category of quasi-coherent $k$-regulous sheaves on $V$ is
exact. \hspace*{\fill} $\Box$
\end{corollary}

Note that every non-empty $k$-regulous open subset $U$ of $K^{n}$
is an affine $k$-regulous variety. Indeed, if $U = \mathcal{D}\,
(f)$ for a $k$-regulous function $f$ on $K^{n}$
(Corollary~\ref{regul-prop-cor-1}), then $U$ is isomorphic to the
affine $k$-regulous subvariety
$$ V := \mathcal{Z}\, (y f(x) -1) \subset K^{n}_{x} \times K_{y}.
$$

\begin{remark}\label{rem-regulous}
Consider a smooth algebraic subvariety $X$ of the affine space
$K\mathbb{A}^{n}$. We may look at the set $V := X(K)$ of its
$K$-points both as an algebraic variety $X(K)$ and as a
$k$-regulous subvariety $V$ of $K^{n}$. Every function $f: V \to
K$ that is $k$-regulous on $V$ in the second sense remains, of
course, $k$-regulous on $X(K)$ in the sense of the definition from
the beginning of Section~11.

\vspace{1ex}

{\bf Open problem.} The problem whether the converse implication
is true for $k>0$ is unsolved as yet.

\vspace{1ex}

For $k=0$ the answer is in the affirmative and follows immediately
from Theorem~\ref{ext-th}. Indeed, every continuous hereditarily
rational function $f$ on $X(K)$ extends to a continuous rational
function on $K^{n}$, whence $f$ is regulous on $V$. This theorem
was proven for real and $p$-adic varieties in~\cite{K-N}.
\end{remark}

Finally, we wish to give a criterion for a continuous function to
be regulous. It relies on Theorem~\ref{ext-th} on extending
continuous hereditarily rational functions on algebraic
$K$-varieties.

\begin{proposition}\label{criterion}
Let $V$ be an affine regulous subvariety of $K^{n}$ and $f: V \to
K$ a function continuous in the $K$-topology. Then a necessary and
sufficient condition for $f$ to be a regulous function is the
following:

\vspace{1ex}

(*) For every Zariski closed subset $Z$ of $K^{n}$ there exist a
Zariski dense open subset $U$ of the Zariski closure of $V \cap Z$
in $K^{n}$ and a regular function $g$ on $U$ such that
$$ f(x) = g(x)\ \ \ \text{for all} \ \ x \in V \cap Z \cap U. $$
\end{proposition}

\begin{proof}
By Corollary~\ref{filtration-cor-2}, the necessary condition is
clear, because $f$ is the restriction to $V$ of a regulous
function on $K^{n}$ (Corollary~\ref{B-cor-1} to Cartan's
Theorem~B).

In order to prove the sufficient condition, we proceed with
induction with respect to the dimension $d$ of the set $V$ which
is a closed (in the $K$-topology) constructible subset of $K^{n}$.
The case $d=0$ is trivial. Assuming the assertion to hold for
dimensions less than $d$, we shall prove it for $d$. So suppose
$V$ is of dimension $d$. By Corollary~\ref{correspond-prop}, the
Zariski closure $W$ of $V$ in $K^{n}$ is of dimension $d$ and we
have
$$ W^{0} := \{ x \in W: \ W \ \text{is smooth of dimension} \ d \
   \text{at} \ x \} \subset V. $$
Obviously, $W^{*} := W \setminus W^{0}$ is a Zariski closed subset
of $K^{n}$ of dimension less than $d$. Therefore $Y := V \cap
W^{*}$ is a regulous closed subset of $V$ of dimension less that
$d$ and
$$ W = W^{0} \cup W^{*} \subset V \cup W^{*} \subset W. $$
Since the restriction $f|Y$ satisfies condition (*), it is a
regulous function on $Y$ by the induction hypothesis. It is thus
the restriction to $Y$ of a regulous function $F$ on $K^{n}$
(Corollary~\ref{B-cor-1} to Cartan's Theorem~B).

Further, the function $f$ and the restriction $F|W^{*}$ can be
glued to a function
$$ g: W = V \cup W^{*} \to K, \ \ \ g(x) =
   \left\{
     \begin{array}{ll}
        f(x)\, : & x \in V \\
        F(x):    & x \in W^{*}
     \end{array}
   \right.
$$
which satisfies condition (*) as well. Now, it follows from
Theorem~\ref{ext-th} that $g$ extends to a regulous function $G$
on $K^{n}$. Since $f$ is the restriction to $V$ of the function
$G$ which is regulous, so is $f$, as required.
\end{proof}

\vspace{3ex}

{\bf Acknowledgements.} The author wishes to express his gratitude
to J\'{a}nos Koll\'{a}r and Wojciech Kucharz for several
stimulating discussions on the topics of this paper.

\vspace{1ex}

%\newpage

\vspace{5ex}

\begin{small}
%\begin{sc}
Institute of Mathematics

Faculty of Mathematics and Computer Science

Jagiellonian University

%Faculty of Mathematics and Computer Science

ul.~Profesora \L{}ojasiewicza 6

30-348 Krak\'{o}w, Poland

{\em e-mail address: nowak@im.uj.edu.pl}
%\end{sc}
\end{small}

\end{document}